\documentclass[12pt,oneside,reqno]{amsart}
\usepackage[margin=1in]{geometry}
\usepackage{color}
\usepackage{tikz-cd}
\usetikzlibrary{arrows.meta, positioning}
\usepackage{esint,amssymb}
\usepackage{graphicx}
\usepackage{MnSymbol}
\usepackage{mathtools}
\usepackage[colorlinks=true, pdfstartview=FitV, linkcolor=blue, citecolor=pink, urlcolor=blue,pagebackref=false]{hyperref}
\usepackage{microtype}
\usepackage{amsmath}
\usepackage{xifthen}
\usepackage{verbatim}
\definecolor{darkgreen}{rgb}{0,0.5,0}
\definecolor{darkblue}{rgb}{0,0,0.7}
\definecolor{darkred}{rgb}{0.9,0.1,0.1}
\usepackage[all]{xy}
\usepackage[makeroom]{cancel}
\usepackage{enumitem}
\usepackage{natbib}
\makeatletter
\def\@tocline#1#2#3#4#5#6#7{\relax
  \ifnum #1>\c@tocdepth % then omit
  \else
    \par \addpenalty\@secpenalty\addvspace{#2}%
    \begingroup \hyphenpenalty\@M
    \@ifempty{#4}{%
      \@tempdima\csname r@tocindent\number#1\endcsname\relax
    }{%
      \@tempdima#4\relax
    }%
    \parindent\z@ \leftskip#3\relax \advance\leftskip\@tempdima\relax
    \rightskip\@pnumwidth plus4em \parfillskip-\@pnumwidth
    #5\leavevmode\hskip-\@tempdima
      \ifcase #1
       \or\or \hskip 1em \or \hskip 2em \else \hskip 3em \fi%
      #6\nobreak\relax
    \dotfill\hbox to\@pnumwidth{\@tocpagenum{#7}}\par
    \nobreak
    \endgroup
  \fi}
\makeatother

\newtheorem{theorem}{Theorem}
\newtheorem{proposition}[theorem]{Proposition}
\newtheorem{lemma}[theorem]{Lemma}
\newtheorem{corollary}[theorem]{Corollary}

\theoremstyle{definition}
\newtheorem{remark}[theorem]{Remark}

\newtheorem{definition}[theorem]{Definition}

\newcommand{\cref}[1]{Corollary~\ref{c.#1}}

\numberwithin{equation}{section}
\numberwithin{theorem}{section}

\newcommand{\Z}{\mathbb{Z}}

\newcommand{\R}{\mathbb{R}}

\newcommand{\C}{\mathbb{C}}

\renewcommand{\subset}{\subseteq}
\newcommand{\h}{\mathfrak{H}}

\newcommand{\eps}{\varepsilon}

\newcommand{\test}[1][]{%
\ifthenelse{\equal{#1}{}}{omitted}{given}%
}

\renewcommand{\d}[1]{\ensuremath{\operatorname{d}\!{#1}}}
\newcommand{\pa}{\partial}
\newcommand{\del}{\partial}

\renewcommand{\part}{\partial}

\usepackage{dsfont}
\usepackage{slashed}

\newcommand{\Tr}{\operatorname{Tr}}

\newcommand{\End}{\operatorname{End}}

\newcommand{\Id}{\operatorname{Id}}
\newcommand{\Hom}{\operatorname{Hom}}

\newcommand{\HK}{\operatorname{HK}}
\renewcommand{\d}{\dagger}
\newcommand{\hb}{\bar{h}}
\renewcommand{\dim}{\operatorname{dim}}
\newcommand{\delbar}{\bar\pa}
\newcommand{\Hod}{\operatorname{Hod}}
\newcommand{\dR}{\operatorname{dR}}

\newcommand{\rk}{\operatorname{rk}}
\newcommand{\GL}{\operatorname{GL}}

\begin{document}
\title{The conformal limit for Nakajima quiver varieties}

\author[Sotiria Chatzimarkou and Panagiotis Dimakis]{Sotiria Chatzimarkou and Panagiotis Dimakis}
\subjclass[2010]{}
\date{\today}
\begin{abstract}
Inspired by Gaiotto's conformal limit construction for Higgs bundles we introduce and study a conformal limit construction for Nakajima quiver varieties. We prove that the conformal limit is indeed a limit of a one-parameter family of points inside a specified quiver variety and that it gives a biholomorphic map between holomorphic Lagrangian submanifolds foliating two related quiver varieties. Finally, we introduce the analog of Simpson's conjecture on the completeness of these holomorphic Lagrangian submanifolds and provide a complete proof under a genericity assumption. Central to both proofs is the existence of a Bia\l{}ynicki-Birula slice associated to fixed points of a natural circle action. 
\end{abstract}

\maketitle

\begin{center}
\textit{In memory of Stavros Chatzimarkos}
\end{center}

\tableofcontents

\section{Introduction}

On a compact Riemann surface $\Sigma$ of genus $g\ge 2$ equipped with a complex vector bundle $E$ of rank $n$ and degree zero one can define the Hodge moduli space $M_{Hod}$ parametrizing polystalbe $\lambda$-connections. $M_{Hod}$ admits a natural $\C^{\star}$--action and a $\C^{\star}$--equivariant map 
\begin{equation*}
\pi:M_{Hod} \to \C.
\end{equation*}
The fiber $\pi^{-1}(0)$ is the moduli space $M_H$ parametrizing polystable Higgs bundles and the fiber $\pi^{-1}(1)$ is the de Rham moduli space $M_{dR}$ parametrizing completely reducible complex flat connections. Given a stable Higgs bundle $(\bar\pa_E,\Phi)$ with harmonic metric $h$, consider the twistor line 
\begin{equation*}
[(\lambda, \bar\pa_E + \lambda\Phi^{\dagger_h}, \lambda\pa_E^{\dagger_h} + \Phi)] \in \pi^{-1}(\lambda). 
\end{equation*}
Combining the $\C^{\star}$--action and the existence of twistor lines we can perform the following set of operations 

\begin{equation}\label{snake}
\begin{tikzcd}
(\bar\pa_E,\Phi) \arrow{r} &
(\bar\pa_E,R\Phi) \arrow{r} &
(\lambda, \bar\pa_E+ \lambda R\Phi^{\dagger_{h_R}}, \lambda\pa^{\dagger_{h_R}}+R\Phi) \arrow{dll} \\
(1, \bar\pa_E+ \lambda R\Phi^{\dagger_{h_R}}, \pa^{\dagger_{h_R}}+\lambda^{-1}R\Phi) \arrow{rr}{\hbar:=\lambda R^{-1}}
& &\nabla_{R,\hbar} = \bar\pa_E+ \pa_E^{\dagger_{h_R}} + \hbar R^2\Phi^{\dagger_{h_R}} +\hbar^{-1}\Phi
\end{tikzcd}
\end{equation}
in order to obtain a family of complex flat connections indexed by $R>0$ and $\hbar\in \C^{\star}$ a constant. It was conjectured by Gaiotto in \cite{Gaiotto} and later proven in \cite{DFKMMN} that when $(\bar\pa_E,\Phi)$ lies in the Hitchin section, the limit 
\begin{equation*}
\lim\limits_{R\to 0 } \nabla_{R,\hbar}
\end{equation*}
exists and provides a biholomorphism from the Hitchin section to the space of Opers in $M_{dR}$. In \cite{CW} this theorem was generalized to the following:
\begin{theorem}
Let $W(\bar\pa_0,\Phi_0)$ be the set of points in $M_{Hod}$ whose $\C^{\star}$--limit is the $\C^{\star}$--fixed point $[(\bar\pa_0,\Phi_0)]$ with harmonic metric $h_0$. Then:
\begin{itemize}
\item the restrictions $W^0(\bar\pa_0,\Phi_0) = W(\bar\pa_0,\Phi_0)\cap\pi^{-1}(0)$ and $W^1(\bar\pa_0,\Phi_0) = W(\bar\pa_0,\Phi_0)\cap \pi^{-1}(1)$ are holomorphic Lagrangian submanifolds of $M_H$ and $M_{dR}$ respectively. 
\item For $[(\bar\pa_E,\Phi)]\in W^0(\bar\pa_0,\Phi_0)$ the conformal limit exists. 
\item There is a particular representative $(\bar\pa_E,\Phi)$ of the equivalence class $[(\bar\pa_E,\Phi)]$ for which the conformal limit takes the following expression
\begin{equation}\label{HCM}
\mathcal {CL}_{\hbar}(\bar\pa_E,\Phi) := \lim\limits_{R\to 0} \nabla_{R,\hbar} = \bar\pa_E + \pa_0^{\dagger_{h_0}} + \hbar\Phi_0^{\dagger_{h_0}} + \hbar^{-1}\Phi.
\end{equation}
\item $\mathcal {CL}_{\hbar}$ is a biholomorphic map from $W^0(\bar\pa_0,\Phi_0)$ to $W^1(\bar\pa_0,\Phi_0)$. 
\end{itemize}
\end{theorem}
Central in the proof is the introduction of a particular set of representatives for the points in $W^0(\bar\pa_0,\Phi_0)$ called the Bia\l{}ynicki-Birula slice. 

More general conformal limit constructions have recently been obtained in \cite{CFW} for parabolic Higgs bundles, in \cite{KW} for Higgs bundles whose $\C^{\star}$--limit is polystable and in \cite{KM} for Higgs bundles lying in certain Cayley components. Apart from its inherent beauty and simplicity, the conformal limit map provides a way to endow the manifolds $W^1(\bar\pa_0,\Phi_0)$ with coordinates, a fact which was essential in the proof in rank two \cite{DS} of Simpson's conjecture \cite{Simpson} that the $W^1(\bar\pa_0,\Phi_0)$ are complete submanifolds of $M_{dR}$. 

Even though the existence of a conformal limit was conjectured only in the realm of Higgs bundles, it is apparent from \eqref{snake} that an analog of the family $\nabla_{R,\hbar}$ can be constructed for any hyperk\"ahler manifold admitting a $\C^{\star}$--action with the extra property that this $\C^{\star}$--action extends to its twistor space. Thus, given such a hyperk\"ahler manifold, it is meaningful to ask whether a conformal limit exists and whether it has certain desirable properties. 

The purpose of the present paper is to carry out the construction of a conformal limit map for a class of hyperk\"ahler manifolds called Nakajima quiver varieties \cite{Na1994}. We chose to focus on Nakajima quiver varieties primarily because they are the only general family of hyperk\"ahler manifolds with geometric features very similar to Higgs bundle moduli spaces and whose structure is relatively well understood. 

Nakajima quiver varieties are constructed as hyperk\"ahler quotients of a vector space $\mathbb M$ parametrizing all framed representations of a doubled quiver. Specifically, $\mathbb M$ is hyperk\"ahler and admits a compatible group action by $G_{\mathbf{v}}$. As a result it admits a hyperk\"ahler moment map 
\begin{equation*}
\mu = (\mu_{\R}, \mu_{\C}): \mathbb M \longmapsto \mathfrak g_{\mathbf{v}}\oplus \mathfrak g_{\mathbf{v}}^{\C}.
\end{equation*}
Given $(\zeta_{\R},\zeta_{\C}) \in Z\oplus(Z\otimes\C)$ where $Z$ is the center of $G_{\mathbf{v}}$, we define the Nakajima quiver variety as 
\begin{equation*}
\mathfrak M_{(\zeta_{\R},\zeta_{\C})}:= \{ \mu_{\R}^{-1}(\zeta_{\R})\cap\mu_{\C}^{-1}(\zeta_{\C})\}/G_{\mathbf{v}}.
\end{equation*}
Under a suitable genericity assumption on $(\zeta_{\R},\zeta_{\C})$, these are complete hyperk\"ahler manifolds and when $\zeta_{\C}=0$ they admit a $\C^{\star}$-action. The analogue of the Hodge moduli space is given by 
\begin{equation*}
M^Q_{Hod} := \bigcup\limits_{\xi\in \C} \mathfrak M_{((1-|\xi|^2)\zeta_{\R}, -2i\xi\zeta_{\R})}
\end{equation*}
and the twistor line for $(B_h,i_k,B_{\bar h},j_k) \in \mu_{\R}^{-1}(\zeta_{\R})\cap\mu_{\C}^{-1}(0)$ is given by 
\begin{equation*} (B_h - \xi B_{\bar h}^{\dagger}, i_k - \xi j_k^{\dagger}, B_{\bar h} + \xi B_h^{\dagger}, j_k+ \xi i_k^{\dagger})\in \mu_{\R}^{-1}((1-|\xi|^2)\zeta_{\R})\cap\mu_{\C}^{-1}(-2i\xi\zeta_{\R}).
\end{equation*}
Therefore, we can build the same family as in \eqref{snake} for quiver varieties 
\begin{equation}\label{snake2}
\begin{tikzcd}
p= (B_h,i_k,B_{\bar h},j_k) \arrow{d} \\
(B_h,i_k,RB_{\bar h},Rj_k) \arrow{d} \\
(B_h - \xi R B_{\bar h}^{\dagger_{h_R}}, i_k - \xi Rj_k^{\dagger_{h_R}}, RB_{\bar h} + \xi B_h^{\dagger_{h_R}}, Rj_k+ \xi i_k^{\dagger_{h_R}}) \arrow{d}{\hbar = \xi R^{-1}} \\
(B_h - \hbar R^2 B_{\bar h}^{\dagger_{h_{R,\xi}}}, i_k - \hbar R^2j_k^{\dagger_{h_{R,\xi}}}, \hbar^{-1}B_{\bar h} +  B_h^{\dagger_{h_{R,\xi}}}, \hbar^{-1}j_k+ i_k^{\dagger_{h_{R,\xi}}}) = F_{R,\hbar}(p). 
\end{tikzcd}
\end{equation}
and ask whether the limit exists and what properties it may possess. 
\begin{remark}
Notice that going from the second to the third and from the third to the fourth line the metric changes. This happens because the real moment map is not preserved by the $\C^{\star}$-action so in both steps the metric changes in order to satisfy the appropriate real moment map. 
\end{remark}
In order to answer the question we follow the paradigm of \cite{CW} and define a \emph{Hodge slice} $S_p$ in the space $\mathbb M$ which via a Kuranishi map provides a local coordinate chart of $\mathfrak M_{(\zeta_{\R},0)}$ about the point $[p]$. If furthermore $[p] = [p_0] = [(B_h^0,i_k^0,B_{\bar h}^0,j_k^0)]$ is a fixed point of the $\C^{\star}$--action then the deformation theory allows us to define a natural subvariety $S^+_{p_0}\subset S_{p_0}$ which we call the Bia\l{}ynicki-Birula (BB) slice. If $W(p_0)$ is the set of points in $M^Q_{Hod}$ whose $\C^{\star}$--limit is the point $p_0$ then we show that the restrictions $W^0(p_0) = W(p_0)\cap\mathfrak M_{(\zeta_{\R},0)}$ and $W^1(p_0) = W(p_0)\cap\mathfrak M_{(0,-2i\zeta_{\R})}$ are holomorphic Lagrangian submanifolds of $M_{(\zeta_{\R},0)}$ and $M_{(0,-2i\zeta_{\R})}$ respectively and that the BB slice is biholomorphic to both of them. In fact we prove the stronger statement that for every $A\in S^+_{p_0}$ there exists a unique $g_A\in G_{\mathbf{v}}^{\C}$ such that $g_A\cdot(p_0+A)\in W^0(p_0)$. This will allow us to prove 
\begin{theorem}\label{main}
Given $[p_0]$ and $A\in S^+_{p_0}$ the conformal limit 
\begin{equation}
\begin{split}
&\mathcal {CL}_{\hbar}(g_A\cdot(p_0+A)):= \lim\limits_{R\to  0}F_{R,\hbar}(g_A\cdot(p_0+A)) =\\ &g_{A,\hbar}\cdot(B_h^0+A_h - \hbar(B_{\bar h}^0)^{\dagger}, i_k^0+ I_k - \hbar(j_k^0)^{\dagger}, \hbar^{-1}(B_{\bar h}^0 + A_{\bar h}) + (B_h^0)^{\dagger}, \hbar^{-1}(j_k^0+J_k)+(i_k^0)^{\dagger})
\end{split}
\end{equation}
is a biholomorphism from $W^0(p_0)$ to $W^1(p_0)$. 
\end{theorem}
\begin{remark}
The point 
\begin{equation}\label{explicit}
p_A:= (B_h^0+A_h - \hbar(B_{\bar h})^{\dagger}, i_k^0+ I_k - \hbar(j_k^0)^{\dagger}, \hbar^{-1}(B_{\bar h}^0 + A_{\bar h}) + (B_h^0)^{\dagger}, \hbar^{-1}(j_k^0+J_k)+(i_k^0)^{\dagger})
\end{equation}
satisfies the complex moment map $\mu_{\C}(p_A) =-2i\zeta_{\R}$ and $g_{A,\hbar}\in G_{\mathbf{v}}^{\C}$ is the unique complex gauge transformation for which the real moment map $\mu_{\R}(g_{A,\hbar}\cdot p_A)=0$ is satisfied. 
\end{remark}
\begin{remark}
By abuse of notation, we sometimes write $\mathcal {CL}_{\hbar}(p_0+A)$ instead of $\mathcal {CL}_{\hbar}(g_A\cdot(p_0+A))$ since $g_A$ is uniquely determined.
\end{remark}
Finally, we present an application. We formulate and prove the analogue of Simpson's foliation conjecture \cite{Simpson} for all Nakajima quiver varieties:
\begin{theorem}
$W^1(p_0)$ are complete submanifolds of $\mathfrak M_{(0,-2i\zeta_{\R})}$ for any fixed point $p_0$ assuming $\zeta_{\R}$ is generic. 
\end{theorem}
Unlike the proof strategy idea developed in \cite{Schulz,DS}, we do not try to obtain explicit growth estimates along rays in $W^1(p_0)$. Our argument combines the results of Le Bruyn and Procesi \cite{LP1990} on finite generation by traces of words of the ring of invariants of quiver varieties with the explicit parametrization \eqref{explicit} and the fact that the conformal limit map is a biholomorphism from $W^0(p_0)$ to $W^1(p_0)$, reducing the problem to the simple fact that non-constant polynomials become unbounded. Despite its simplicity, we believe that this argument provides a conceptual shift in how to approach similar questions in hyperk\"ahler geometry. 
\\
\paragraph{\textbf{Aknowledgements}} The authors would like to thank Georgios Kydonakis for many fruitful discussions and for carefully proofreading this manuscript. The second author would like to thank Georgios Kydonakis for providing the opportunity for an extended visit at the University of Patras during which a significant portion of this work was completed. 

\section{Preliminaries}

In the first part of this section we introduce Nakajima quiver varieties and present some of their basic properties, necessary for the later parts of the paper \cite{Na1994}. In the second part of the section we explain the core ideas behind the proof that the conformal limit for Higgs bundles exists \cite{CW} that we will generalize to the quiver variety setting.  

\subsection{The geometry of Nakajima quiver varieties}

Let $\Gamma$ be a finite, non-empty graph with $n$ vertices and let $H$ be the set of pairs consisting of an edge together with an orientation. From here on, we fix an orientation of the graph $\Gamma$. By this we mean a subset $\Omega \subset H $ such that $\bar \Omega \cup \Omega =H$,  $\Omega\cap\bar\Omega=\emptyset$. For $h \in H$, we denote $\text{in}(h)$ the incoming vertex and $\text{out}(h)$ the outgoing vertex. We also denote $\hb$ the edge with its orientation reversed, meaning that $\text{in}(\hb)=\text{out}(h)$ and $\text{out}(\hb)=\text{in}(h)$. Two different vertices may be joined by several edges, but no edge may join a vertex to itself (no loops are allowed). We number the vertices and identify the set of vertices with the set ${1,2,...,n}$. Let $(V_k, W_k)$ be pairs of hermitian vector spaces for every vertex $k \in {1,2,...n}$. 

\begin{definition}

We define the \textbf{Nakajima quiver representation space}:
\begin{equation}\label{NQRS}
   \mathbf M=\mathbf M(\mathbf{v},\mathbf{w}):=\left(\bigoplus_{h\in H} \Hom(V_{\text{out}(h)},V_{\text{in}(h)})\right)\oplus\left(\bigoplus_{k=1}^n \Hom(W_k,V_k)\oplus \Hom(V_k,W_k)\right),
\end{equation}
where $ \mathbf{v}= (\dim_\C V_1,...,\dim_\C V_n)$ and $\mathbf{w}=(\dim_\C W_1,...,\dim_\C W_n)$ denote the dimensions of the hermitian vector spaces.
We denote the elements of $\mathbf M$ by triples $(B,i,j)$ where $B_h \in \Hom(V_{\text{out}(h)},V_{\text{in}(h)})$ with $B=(B_h)_{h \in H}$, $i_k \in \Hom(W_k,V_k)$ with $i=(i_k)_{1\le k\le n}$ and $j_k \in \Hom(V_k,W_k)$ with $j=(j_k)_{1\le k\le n}$. 
\end{definition}
The space $\mathbf{M}$ carries a natural symplectic form $\omega_\C$ given by:
\begin{equation}\label{CSF}
    \omega_\C((B,i,j),(B',i',j')):=\sum\limits_{h \in H} \Tr (\eps(h)B_h B_{\hb})+\sum\limits_{k=1}^n \Tr(i_k j'_k-i'_kj_k),
\end{equation}
where $\eps(h)=1$ if $h \in \Omega$ and $\eps(h)=-1$ if $h \in \bar \Omega$.

The symplectic vector space $\mathbf{M}$ decomposes into the sum of Lagrangian subspaces $\mathbf{M}=\mathbf{M_\Omega}\oplus \mathbf{M_{\bar \Omega}}$ where:
\begin{equation*}
    \begin{split}
        \mathbf{M_\Omega}:=&\bigoplus _{h \in \Omega}\Hom (V_{\text{out}(h)},V_{\text{in}(h)})\oplus \bigoplus_ {k=1}^n \Hom(W_k,V_k) \\
        \mathbf{M_{\bar \Omega}}:=&\bigoplus _{h \in \bar \Omega}\Hom (V_{\text{out}(h)},V_{\text{in}(h)})\oplus \bigoplus _{k=1}^n \Hom(V_k,W_k).
    \end{split}
\end{equation*}
Since $V,W$ are hermitian vector spaces there exists a hermitian inner product on $\Hom(V,W)$ defined by $(f,g)=\Tr(fg^\d)$, where $(\bullet)^\d$ denotes the hermitian adjoint. Using this inner product we define a second complex structure $J$ on $\mathbf{M}$ given by:

\begin{equation*}
J(m,m')=({-m'}^\d,m^\d) \text{ for } (m,m')\in\mathbf{M_\Omega}\oplus \mathbf{M_{\bar \Omega}}.
\end{equation*}
Combining the symplectic form $\omega_{\C}$ and the complex structure $J$ we define the metric 
\begin{equation}\label{Qmetric}
g((B,i,j),(B',i',j')) := \omega_{\C}((B,i,j),J\cdot(B',i',j')) = \sum\limits_{h\in H} \Tr(B_hB_h^{'\dagger}) + \sum\limits_{k=1}^n \Tr(i_k i_k^{\dagger} + j_k^{'\dagger}j_k).
\end{equation}

The two complex structures define a hyperk\"ahler structure on $\mathbf{M}$. The gauge group $G_{\mathbf{v}}=\prod_{k=1}^{n} \mathrm{U}(\dim V_k)$ acts on $\mathbf{M}$ by:
\begin{equation*}
   g \cdot  (B_h,i_k,j_k)= (g_{\text{in}(h)}B_h g_{\text{out}(h)}^{-1},g_ki_k,j_kg_k^{-1}),
\end{equation*}
preserving the hyperk\"ahler structure and the hermitian inner product on $\mathbf{M}$.
Let $\mu $ be the corresponding hyperk\"ahler moment map and $\mu_\R,\mu_\C$ the real and complex components of $\mu$. They have the following explicit forms:
\begin{equation}
\begin{split}
    \mu_\R (B,i,j)&=\frac{i}{2}\left(\sum \limits_{\substack{\{ h \in H\\ k=in(h)\}}} B_hB_h^\d-B_{\hb}^\d B_{\hb}+i_ki_k^\d -j_k^\d j_k\right)_k \in \bigoplus _k \mathfrak{u}(V_k)=\mathfrak{g}_{\mathbf{v}} \\
    \mu_\C(B,i,j)&=\left(\sum \limits_{\substack{\{ h \in H\\ k=in(h)\}}} \eps(h)B_hB_{\hb}+i_kj_k \right)_k \in \bigoplus _k \mathfrak{gl}(V_k)=\mathfrak{g}_{\mathbf{v}}\otimes \C,
\end{split}
\end{equation}
where $\mathfrak{g}_{\mathbf{v}}$ is the Lie algebra of $G_{\mathbf{v}}$. Above we have identified $\mathfrak{g}$ with its dual $\mathfrak{g^*}$.
\begin{definition}
    Let $Z\subset \mathfrak{g}_{\mathbf{v}}$ be the center of the Lie algebra $\mathfrak{g}_{\mathbf{v}}$ and let $\zeta =(\zeta_\R,\zeta_\C)$ be an element of $Z\oplus (Z\otimes \C)$. We define the Nakajima quiver variety as the hyperk\"ahler quotient $\mathfrak{M}_\zeta$ of $\bold M$ by $G_{\mathbf{v}}$:
    \begin{equation}
        \mathfrak{M}_\zeta =\mathfrak{M}_{\zeta}(\mathbf v, \mathbf w):= \{(B,i,j) \in \mathbf M | \mu (B,i,j)=\zeta\}/G_{\mathbf{v}}.
    \end{equation}
\end{definition}
We denote by $[(B,i,j)]$ the $G_{\mathbf{v}}$-orbit of $(B,i,j)$ considered as a point on $\mathfrak{M}_\zeta$.

Depending on the choice of $\zeta$ the space $\mathfrak{M}_{\zeta}$ may have singularities. The subset
\begin{equation}
\mathfrak{M}_{\zeta}^{reg}=\{(B,i,j) \in \mu^{-1}(-\zeta)|\text{ the stabilizer of }(B,i,j) \text{ in } G_{\mathbf{v}} \text{ is trivial}\}/G_{\mathbf{v}}
\end{equation}
is a generally incomplete nonsingular hyperk\"aler manifold with dimension: 
\begin{equation*}
\dim_\R \mathfrak{M_\zeta ^{reg}}=\dim_\R M-\dim_\R G_v=2^t \mathbf{v}(2\mathbf{w}-\mathbf{Cv}), 
\end{equation*}
where $C=2I-A$ is the generalized Cartan matrix. The natural symplectic form $\omega_{\C}$ descends to define a symplectic form on $\mathfrak{M}_\zeta ^{reg}$ which we also denote by $\omega_\C$. Now, let: 
\begin{equation*}
    \begin{split}
        R_+ &=\{\theta =(\theta_k) \in \Z_{\ge 0}^n | ^t \theta\mathbf{C}\theta \le2\}\backslash \{0\} \\
        R_+(\mathbf{v})&= \{\theta \in R_+|\theta _k \le \dim_\C V_k \text{ for every } k\} \\
        D_{\theta}&=\{x=(x_k) \in \R^n|\sum \limits_k x_k \theta _k =0\}.
    \end{split}
\end{equation*}
\begin{theorem}[Theorem $2.8$ in \cite{Na1994}]\label{generic}
    Suppose 
    \begin{equation}\label{generic}
        \zeta \in \R^3\oplus\R^n \backslash\bigcup \limits_ {\theta \in R_+(\mathbf{v})}\R^3\oplus D_\theta
    \end{equation}  
    then the regular locus $\mathfrak{M}_\zeta ^{reg}$ coincides with $\mathfrak{M}_\zeta$ thus making $\mathfrak{M}_{\zeta}$ a complete hyperk\"ahler manifold.
\end{theorem}
\begin{definition}
    We call $\zeta$ \textbf{generic} if it satisfies condition \eqref{generic}.
\end{definition}
In order to define the Hodge and the Bia\l{}ynicki-Birula slice below we need to also define the Nakajima quiver variety as a GIT quotient by the complexified gauge group $G_{\mathbf{v}}^{\C}$. 
\begin{definition}
    Fix $\zeta=(\zeta_\R,\zeta_\C)$. The set of stable points is defined as:
    \begin{equation*}
        H_\zeta ^s=\{m \in \mu_\C^{-1}(\zeta_\C)~|\text{ the } G_{\mathbf{v}}^\C \text {-orbit through } m \text{ intersects the  set } \mu_\R^{-1}(\zeta_\R)\}.
    \end{equation*}
\end{definition}
\begin{theorem}[Theorem $3.2$ in \cite{Na1994}]
    Let $\zeta =(\zeta_\R,\zeta_\C)$ generic, then: \\
    i) we have a homeomorphism 
    \begin{equation*}
        \mathfrak{M}_\zeta \cong H_\zeta ^s/G_v^\C,
    \end{equation*}
    ii) $\mu _\C ^{-1}(-\zeta_\C)\backslash H_\zeta^s $ is a complex subvariety of $\mu_\C^{-1}(-\zeta_\C).$
\end{theorem}

\subsection{The conformal limit for Higgs bundles}
In order to discuss the conformal limit we need to introduce the notions of Higgs bundle and $\lambda$-connection. For this we need to fix a compact Riemann surface $\Sigma$ of genus $g\ge 2$ equipped with a complex vector bundle $E$ of rank $n$ and degree zero.
\begin{definition}\label{stabHB}
    A \textbf{Higgs bundle} is a pair $(\bar\pa_E, \Phi)$ consisting of a Dolbeault operator $\bar\pa_E$ which induces a holomorphic structure on $E$ and a $1$-form $\Phi\in H^0(\Sigma,\End( E)\otimes K)$ called the \emph{Higgs field} satisfying the compatibility condition 
    \begin{equation*}
    \bar\pa_E\Phi =0
    \end{equation*}
    i.e., the Higgs field is required to be holomorphic. Here $K$ denotes the canonical line bundle over $\Sigma$. 
\end{definition}
\begin{definition}
    A Higgs bundle $(\bar\pa_E,\Phi)$ is
    \begin{itemize}
    \item \textbf{stable} if for every non-zero proper $\Phi$-invariant holomorphic subbundle $F\subset E$, we have:
    \begin{equation}
    \frac{\deg(F)}{\rk (F)} < \frac{\deg(E)}{\rk(E)}. 
    \end{equation}
    \item \textbf{poly-stable} if it is a direct sum of stable Higgs bundles.
    \end{itemize}
\end{definition}
Let $\mathcal H^{ps}$ denote the set of poly-stable Higgs bundles. The complex gauge group $\mathcal{G}^\C=\Gamma(\Sigma,\GL(E))$ acts naturally on $\mathcal H^{ps}$:
\begin{equation*}
    g\cdot (\bar \del _E,\Phi):=(g^{-1}\circ\bar \del_E \circ g, g^{-1}\circ \Phi \circ g).
\end{equation*}
We define the moduli space of Higgs bundles as equivalence classes of polystable Higgs bundles under the gauge group action $M_H = \mathcal H^{ps}/\mathcal G^{\C}$. 

The moduli of Higgs bundles $M_H$ fits into a larger family of objects known as $\lambda$-connections which we now define. 

\begin{definition}
Given $\lambda \in \C$, a \textit{$\lambda$-connection} on $E$ is a triple $(\lambda, \delbar_E , \nabla_\lambda)$ where $\delbar_E$ is a Dolbeault operator on $E$ and $\nabla_\lambda : \Omega^0 (E) \rightarrow \Omega^{1,0}(E)$ is a differential operator, subject to:
\begin{enumerate}[label = (\roman*)]
\item a $\lambda$-twisted Leibniz rule: 
    
for all $f \in C^\infty (\Sigma)$ and $s \in \Omega^0 (E)$:
\begin{equation*}
        \nabla_\lambda (fs) = \lambda \partial f \otimes s + f \nabla_\lambda(s),
\end{equation*} 
    \item a holomorphicity condition $[ \delbar_E , \nabla_\lambda ] = 0$.
\end{enumerate}
\end{definition}
In order to obtain a geometrically well behaved moduli space we restrict ourselves to $\lambda$-connections satisfying a stability condition similar to definition \ref{stabHB}:
\begin{definition}
A $\lambda$-connection $(\lambda, \bar\pa_E, \nabla_{\lambda})$ is 
\begin{itemize}
    \item \textit{stable} if for any $\nabla_{\lambda}$-invariant subbundle $\mathcal F\subset \mathcal E$ it holds that $\deg \mathcal F < 0$,
    \item \textit{polystable} if it is a direct sum of stable $\lambda$-connections.
\end{itemize}
\end{definition}
We denote the space of polystable $\lambda$-connections by $\mathcal M^{ps}$. It admits an action by the gauge group $\mathcal G^{\C}$ given by 
\begin{equation*}
    g\cdot (\lambda,\bar \del _E,\nabla_{\lambda}):=(\lambda,g^{-1}\circ\bar \del_E \circ g, g^{-1}\circ \nabla_{\lambda} \circ g).
\end{equation*}
\begin{definition}
We define the \textbf{Hodge moduli space} $M_{Hod}:= \mathcal M^{ps}/\mathcal G_{\mathbb C}$. 
\end{definition}
The Hodge moduli space admits a natural $\mathbb C^{\star}$-action  
\begin{equation}\label{CSA}
    \xi\cdot [(\lambda,\bar\pa_E,\nabla_{\lambda})]:= [(\xi\lambda,\bar\pa_E,\xi\nabla_{\lambda})],
\end{equation}
as well as a $\C ^\star$-equivariant holomorphic map to $\C$ given by 
\begin{equation*}
  \pi: [(\lambda, \bar\pa_E,\nabla_{\lambda})]\to \lambda .  
\end{equation*}

\begin{remark} When $\lambda = 0$, $\Phi := \nabla_0$ becomes a holomorphic $C^\infty$-linear endomorphism of $(E, \delbar_E)$ with values in $(1,0)$-forms. As we have seen above these are exactly Higgs bundles. When $\lambda = 1$, $D:= \bar\pa_E + \nabla_1$ is a complex flat connection and their moduli space, called the \textit{de Rham moduli space} is denoted as $M_{dR}$. When $\lambda \neq 0$, then $\lambda^{-1}\cdot[(\lambda,\bar\pa_E,\nabla_{\lambda})] \in M_{dR}$. In particular, $\pi ^{-1} (\lambda) \simeq M_{dR}$.
\end{remark}

Note that $M_H^s:= \pi^{-1}(0)\cap\mathcal M^s/\mathcal G_{\mathbb C}$ and $M_{dR}^s := \pi^{-1}(1)\cap\mathcal M^s/\mathcal G_{\mathbb C}$ are holomorphic symplectic manifolds and that they are diffeomorphic through the Nonabelian Hodge correspondence:

\begin{theorem}[The non-abelian Hodge correspondence]
There exists a bijective correspondence between the moduli space $\mathcal M^{ps}_H$ of polystable Higgs bundles and the space of equivalence classes of completely reducible complex flat connections on $E$
\begin{equation*}
    \mathcal M_H^{ps}\cong M_{dR},
\end{equation*}
constructed as follows. A Higgs bundle $(\bar\pa_E,\Phi)$ is polystable if and only if there exists a hermitian metric $h$ on $E$ that satisfies the Hitchin equation: 
\begin{equation*}\label{Hitchineq}
F_{(\bar\pa_E,\Phi)} + [\Phi,\Phi^{\star_h}] = 0,
\end{equation*}
where $F_{(\bar\pa_E,\Phi)}$ is the curvature of the Chern connection $\bar\pa_E+ \pa_E^h$ associated to the pair $(\bar\pa_E,\Phi)$ and $\Phi^{\star_h}$ is the adjoint of $\Phi$ with respect to $h$. We call such a metric \emph{harmonic}. Given such a harmonic metric it turns out that the connection 
\begin{equation*}
D = \bar\pa_E+ \pa_E^h + \Phi+ \Phi^{\star_h}
\end{equation*}
is a completely reducible flat connection which is irreducible if and only if $(\bar\pa_E,\Phi)$ is stable. Conversely, a flat connection is completely reducible if and only if it comes from a polystable Higgs bundle through the above procedure. 
\end{theorem}

We are ready to describe the conformal limit. Let $(\bar\pa_E,\Phi)$ be a stable Higgs bundle. Act by the $\C^{\star}$-action to obtain the stable Higgs bundle $(\bar\pa_E,R\Phi)$ for any $R>0$. For $\xi \in \R$ consider the $\lambda$-connection 
\begin{equation*}
(\xi, \bar\pa_E+ \xi R\Phi^{\dagger_{h_R}}, \xi\pa_E^{\dagger_{h_R}}+ R\Phi).
\end{equation*}
The equivalence class $[(\xi, \bar\pa_E+ \xi R\Phi^{\dagger_{h_R}}, \xi\pa_E^{\dagger_{h_R}}+ R\Phi)] \in \pi^{-1}(\xi)$ is part of the real twistor line in $M_{\Hod}$ passing through $(\bar\pa_E,\phi)$. Act by the $\C^{\star}$-action to obtain 
\begin{equation*}
\xi^{-1}\cdot(\xi, \bar\pa_E+ \xi R\Phi^{\dagger_{h_R}}, \xi\pa_E^{\dagger_{h_R}}+ R\Phi)= (1,\bar\pa_E+ \xi R\Phi^{\dagger_{h_R}}, \pa_E^{\dagger_{h_R}}+ \xi^{-1}R\Phi).
\end{equation*}
The term on the right represents a particular element in $M_{\dR}$; in particular the connection 
\begin{equation*}
\bar\pa_E + \pa_E^{\dagger_{h_R}} + \xi^{-1}R\Phi + \xi R\Phi^{\dagger_{h_R}}
\end{equation*}
is flat. Let $\hbar := R^{-1}\xi$ and write 
\begin{equation*}
D_{R,\hbar} = \bar\pa_E + \pa_E^{\dagger_{h_R}} + \hbar^{-1}\Phi +  \hbar R^2\Phi^{\dagger_{h_R}}.
\end{equation*}
The limiting flat connection $\lim\limits_{R\to 0} D_{R,\hbar}$, when it exists, is a map from the moduli space of Higgs bundles $M_H$ to the de Rham moduli space $M_{dR}$ called the $\hbar$-conformal limit. Given a \textit{stable} Higgs pair $[(\bar\pa_E,\Phi)]$, let $[(\bar\pa_0,\Phi_0)]:=\lim\limits_{\xi\to 0}\xi\cdot [(\bar\pa_E,\Phi)]$ be its $\mathbb C^{\star}$-limit. Then \cite{CW} shows that there is a gauge in which the $\hbar$-conformal limit takes the form 
\begin{equation}
    \mathcal {CL}_{\hbar}(\bar\pa_E,\Phi) = \hbar^{-1}\Phi + \bar\pa_E + \pa_0^{\dagger_{h_0}} + \hbar\Phi_0^{\dagger_{h_0}}.
\end{equation}
Here, $h_0$ is the harmonic metric for the Higgs bundle $(\delbar_0, \Phi_0)$. 
The $\hbar$-conformal limit is discontinuous when its domain is the full moduli space of stable Higgs bundles $M_H^s\subset M_H$. However, when restricted to appropriately defined holomorphic Lagrangian subspaces, it becomes a biholomorphism. Specifically, given $[(\bar\pa_0, \Phi_0)]\in M_H^{\mathbb C^{\star}}$ we define 
\begin{equation*}
    W(\bar\pa_0, \Phi_0):= \{ [(\lambda, \bar\pa_E,\nabla_{\lambda})]\in M_{Hod}: \lim\limits_{\xi\to 0} \xi\cdot [(\lambda, \bar\pa_E,\nabla_{\lambda})] = [(\bar\pa_0, \Phi_0)]\}
\end{equation*}
Let $W^i(\bar\pa_0, \Phi_0):= W(\bar\pa_0, \Phi_0)\cap\pi^{-1}(i)$ for $i\in\{0,1\}$. Then it holds that
\begin{theorem}[Theorem 1.1 in \cite{CW}]\label{CWtheorem}
    $W^i(\bar\pa_0, \Phi_0)\in \pi^{-1}(i)$ are holomorphic Lagrangian submanifolds. The $\hbar$-conformal limit maps $W^0(\bar\pa_0, \Phi_0)$ biholomorphically onto $W^1(\bar\pa_0, \Phi_0)$.
\end{theorem}
Arguably the biggest novelty in \cite{CW} is the introduction of the Bialynicki-Birula slice $S^+_{(\bar\pa_0,\Phi_0)}$ associated to a $\C^{\star}$-fixed point $[(\bar\pa_0,\Phi_0)]$. This is a finite dimensional subvariety of $\mathcal H^{ps}$ which is diffeomorphic to $W^0(\bar\pa_0,\Phi_0)$. In fact it should be thought of as a non-linear analog of the tangent space of the upward flow $W^0(\bar\pa_0,\Phi_0)$ at the origin. The crucial observation is that for any point $[(\bar\pa_E,\Phi)]\in W^0(\bar\pa_0,\Phi_0)$ there is a representative of $[(\bar\pa_E,\Phi)]$ in its equivalence class which is of the form $(\bar\pa_0+\beta,\Phi_0+\phi)$ with $(\beta,\phi)\in S^+_{(\bar\pa_0,\Phi_0)}$. We briefly explain how this object is defined in the Higgs bundle case in order to motivate the similar definitions in the quiver varieties setting. 

Let $(\bar \del_E,\Phi) $ a stable Higgs bundle with Chern connection $\bar \del_E+\del_E$, then the connection $D=\bar \del _E +\del_E+\Phi +\Phi^*$ is flat. Write $D=D'+D''$ where $D'=\del_E+\Phi^*$ and $D''=\bar\del_E+\Phi$. The authors in \cite{CW} start by looking at the deformation complex at a stable Higgs bundle:
\begin{equation}\label{DComplex}
C(\bar\pa_E,\Phi): \Omega^0(\mathfrak{sl}(E))\xrightarrow{D''}\Omega^{0,1}(\mathfrak{sl}(E))\oplus\Omega^{1,0}(\mathfrak{sl}(E))\xrightarrow{D''}\Omega^{1,1}(\mathfrak{sl}(E)).
\end{equation}
Since the operators satisfy the K\"ahler  identities:
\begin{equation*}
(D'')^*=-i[\Lambda,D'], (D')^*=i[\Lambda,D'']
\end{equation*}
we can conclude that $\ker(D'')^*=\ker(D')$. The harmonic representation of the tangent space $H^1(C(\bar\pa_E,\Phi))$ is given by
\begin{equation*}
\begin{split}
\mathcal H^1(\bar\pa_E,\Phi):= \{ (\beta,\phi)\in \Omega^{0,1}(\mathfrak{sl}(E))\oplus\Omega^{1,0}(\mathfrak{sl}(E)) | &\\
D''(\beta,\phi) = 0,&~ D'(\beta,\phi)=0\}.
\end{split}
\end{equation*}
The right arrow in \eqref{DComplex} is given by the linearization of the Higgs bundle equation $\bar\pa_E\Phi=0$. In order to define the \emph{Hodge slice} one replaces the linearization with the actual Higgs bundle equation:
\begin{definition}
     The Hodge slice is given by:
     \begin{equation}
         S_{(\bar \del_E,\Phi)}:=\{(\beta,\phi) \in \Omega^{(0,1)}(\mathfrak{sl}(E))\oplus \Omega^{(1,0)}(\mathfrak{sl}(E))| D''(\beta,\Phi)+[\beta,\Phi]=0,D'(\beta,\Phi)=0\}.
     \end{equation}
\end{definition}
This small change is crucial for the following reason. $S_{(\bar \del_E,\Phi)}$ is a subvariety of $\mathcal H^{ps}$ as opposed to $\mathcal H^1(\bar\pa_E,\Phi)$. Next, the authors of \cite{CW} observe that at a $\C^{\star}$-fixed point the induced $\C^{\star}$-action on $S_{(\bar \del_E,\Phi)}$ introduces a grading compatible with the mapping properties of $D''$ and $D'$. This allows them to restrict to $(\beta,\phi)$ which are in the positive and non-negative pieces of the grading respectively and thus define 
\begin{definition}
     The Bia\l{}inicky-Birula slice is given by:
     \begin{equation}
\begin{split}
S^+_{(\bar\pa_0,\Phi_0)}:= \{ (\beta,\phi)\in \Omega^{0,1}(N^+)\oplus\Omega^{1,0}(L\oplus N^+) ~|~ 
D''(\beta,\phi)+[\beta,\phi]  = 0,D'(\beta,\phi)=0\}.
\end{split}
\end{equation}
\end{definition}

\section{A Hodge moduli space for Nakajima quiver varieties}

 \subsection{Hodge-Nakajima moduli space} In order to effectively define the core objects appearing in the conformal limit for Higgs bundles it was meaningful to embed the Higgs bundle moduli space into the Hodge moduli space $M_{\Hod}$ parametrizing polystable $\lambda$-connections. In this section we construct a similar space for quiver representations and identify core objects inside it. Throughout we fix a real parameter $\zeta_{\R}$ so that 
\begin{equation*}
(\zeta_{\R},0) \in Z_{\mathbf{v}}\oplus(Z_{\mathbf{v}}\otimes\C)
\end{equation*}
is generic \eqref{generic}.
\begin{definition}
We define the \textbf{Hodge-Nakajima moduli space} to be the analytic space 
\begin{equation}
\mathcal M_{\zeta_{\R}} := \bigcup\limits_{\xi \in \C} \mathfrak M_{((1-|\xi| ^2)\zeta_{\R},-2i\xi \zeta_{\R})}.
\end{equation}
\end{definition}
\begin{remark}
The above definition comes from restricting the twistor space of the quiver variety. We found the description of the twistor space in Section $4.3$ of \cite{HKR} to be a very clear so we have chosen to omit it from our paper. 
\end{remark}
As with the Hodge moduli space for Higgs bundles in order to even attempt a definition of the conformal limit we need an explicit description of the twistor line (hyperk\"ahler rotation) associated to a point $(B_h,i_k,B_{\bar h},j_k) \in \mathfrak M_{(\zeta_\R,0)}$. For $\xi \in \C$  the twistor line is given by:
\begin{equation}\label{hrotation}
     (B_h,i_k,j_k) \longmapsto (B_h-\epsilon(h)\xi B_{\bar{h}}^\dagger,i_k-\xi j_k^\dagger,j_k+\xi i_k).
\end{equation}
In the following lemma we check that indeed the twisting by $\xi$ maps points from $\mathfrak M_{(\zeta_{\R},0)}$ to $\mathfrak M_{((1-|\xi| ^2)\zeta_{\R},-2i\xi \zeta_{\R})}$. 
\begin{lemma}
    If the point $q=(B_h,i_k,B_{\bar{h}},j_k )$ satisfies the equations:
\begin{equation}\label{initial}
\begin{split}
    \mu_\R(q) &=\zeta_\R \\
    \mu_\C(q) &=0
    \end{split}
    \end{equation}
    then the hyperk\"ahler rotation $q_\xi=\HK_\xi(q)$ satisfies the equations:
    \begin{equation}\label{hr}
    \begin{split}
        \mu_\R(q_\xi) &=(1-\lvert\xi \rvert^2)\zeta_\R \\
        \mu_\C(q_\xi) &= -2i\xi\zeta_\R  
    \end{split}       
    \end{equation}
\end{lemma}
\begin{proof} First we calculate the real moment map $\mu_\R(q_\xi)$: 
   
\begin{equation*}
\begin{split}
\mu_\R(q_\xi)_k=&\frac{i}{2}\sum\limits_{\substack{\{k=\text{in}(h)\\ h\in H\}}} \Big((B_h-\epsilon(h)\xi B_{\bar{h}}^\dagger)(B_h-\epsilon(h)\xi B_{\bar{h}}^\dagger)^\dagger-(B_{\bar{h}}+\epsilon(h)\xi B_h^\dagger)^\dagger (B_{\bar{h}}+\epsilon(h)\xi B_h^\dagger) \\ 
&+(i_k-\xi (j_k)^\dagger)(i_k-\xi (j_k)^\dagger)^\dagger-(j_k^\dagger +\xi i_k)^\dagger(j_k^\dagger +\xi i_k)\Big)= \\
=& \frac{i}{2} \sum\limits_{\substack{\{k=\text{in}(h)\\ h\in H\}}} \Big((B_h-\epsilon(h)\xi B_{\bar{h}}^\dagger)(B_h^\dagger -\epsilon (h) \bar\xi B_{\bar{h}})-(B_{\bar{h}}^\dagger+\epsilon(h)\bar\xi B_h )(B_{\bar{h}}+\epsilon(h) \xi B_h^\dagger )+\\
&+(i_k -\xi j_k^\dagger)(i_k^\dagger-\bar\xi j_k) -(j_k^\dagger+\bar\xi i_k)(j_k +\xi i_k^\dagger)\Big)= \\
=& \frac{i}{2}\sum\limits_{\substack{\{k=\text{in}(h)\\ h\in H\}}} \Big(B_h B_h^\dagger -\epsilon(h)\bar \xi B_hB_{\bar{h}}-\epsilon(h)\xi B_{\bar{h}}^\dagger B_h^\dagger +\epsilon ^2 (h)\xi \bar \xi B_{\bar{h}}^\dagger B_{\bar{h}}-B_{\bar{h}}^\dagger B_{\bar h} - \epsilon (h) \xi B_{\bar{h}}^\dagger B_h ^\dagger - \\
&- \epsilon^2 (h)\xi \bar \xi  B_h B_h^\dagger -\epsilon (h) \bar \xi B_hB_{\bar h}+i_k i_k^\dagger -\bar \xi i_kj_k -\xi j_k^\dagger i_k^\dagger +\xi \bar \xi  j_k^\dagger j_k -j_k^\dagger j_k -\xi j_k^\dagger i_k^\dagger -\bar \xi i_k j_k -\xi \bar \xi i_k i_k^\dagger\Big)= \\
=& \frac{i}{2}\sum\limits_{\substack{\{k=\text{in}(h)\\ h\in H\}}} \Big(B_h B_h^\d -\epsilon (h) \bar \xi B_h B_{\bar{h}}-\epsilon (h)\xi B_{\bar{h}}^\dagger B_h ^\dagger +\lvert\xi\rvert ^2 B_{\bar{h}}^\dagger B_{\bar{h}} -B_{\bar{h}}^\dagger B_{\bar{h}}-\epsilon (h) \xi B_{\bar{h}}^\dagger B_h^\dagger - \\
&-\epsilon (h) \bar \xi B_h B_{\bar{h}}-\lvert\xi\rvert^2 B_h B_h^\dagger +i_k i_k^\d - \bar \xi i_kj_k -\xi j_k^\d i_k^\d +\lvert\xi\rvert ^2 j_k^\d j_k -j_k^\d j_k -\xi j_k^\d i_k^\d -\bar \xi i_kj_k -\lvert\xi\rvert^2 i_ki_k^\d \Big)= \\
=& \frac{i}{2}\sum\limits_{\substack{\{k=\text{in}(h)\\ h\in H\}}} \Big((1-\lvert\xi\rvert ^2 )B_hB_h^\d -(1-\lvert\xi\rvert ^2) B_{\bar{h}}^\d B_h ^\d +(1-\lvert\xi\rvert ^2) i_k i_k^\d -(1-\lvert\xi\rvert ^2)j_k^\d j_k - \\
&- \epsilon (h) (2\bar \xi B_h B_{\bar{h}} +2\xi B_{\bar{h}}^\d B_h ^\d )-2\bar \xi i_kj_k - 2\xi j_k^\d i_k^\d \Big) = \\
=& \frac{i}{2}(1-\lvert\xi\rvert ^2)\sum\limits_{\substack{\{k=\text{in}(h)\\ h\in H\}}} (B_h B_h ^\d -B_{\bar{h}}^\d B_h ^\d +i_k i_k^\d-j_k^\d j_k) \\
-& i\bar \xi \left(\sum\limits_{\substack{\{k=\text{in}(h)\\ h\in H\}}} \epsilon (h)B_h B_{\bar{h}}+i_k j_k\right) -i\xi\left(\sum\limits_{\substack{\{k=\text{in}(h)\\h\in H\}}}\epsilon(h)B_{\bar{h}}^\d B_h^\d +j_k^\d i_k^\d \right).
\end{split}
\end{equation*}
From equations \eqref{initial} it follows that:

\begin{equation*}
    \begin{split}
        \mu_\R(q_\xi)
=& (1-\lvert\xi\rvert ^2)\mu_\R (q)-\bar \xi{i}\mu_\C (q) -\xi{i}\mu_\C (q)^{\d}= \\
=&(1-\lvert\xi\rvert ^2) \mu _\R (q) =(1-\lvert\xi\rvert ^2) \zeta _\R 
    \end{split}
\end{equation*}
Next, we calculate the complex moment map $\mu_\C(q_\xi)$:
\begin{equation*}
 \begin{split}
\mu _\C (q_\xi )_k=& \left(\sum\limits_{\substack{\{k=\text{in}(h)\\ h\in H\}}} \epsilon (h) (B_h-\epsilon (h) \xi B_{\bar{h}} ^\d) (B_{\bar{h}}+\epsilon(h) \xi B_h^\d)\right)+(i_k-\xi j_k^\d)(j_k+\xi i_k^\d)= \\
=&\left(\sum\limits_{\substack{\{k=\text{in}(h)\\ h\in H\}}}\epsilon(h) (B_h B_{\bar{h}}+\epsilon(h) \xi B_h B_{h}^\d -\epsilon(h) \xi B_{\bar{h}}^\d B_{\bar{h}}-\epsilon^2 (h) \xi ^2 B_{\bar{h}}^\d B_h ^\d )\right)+\\
+&i_k j_k +\xi i_k i_k^\d -\xi j_k ^\d j_k -\xi ^2 j_k ^\d i_k^\d = \\
=& \left(\sum\limits_{\substack{\{k=\text{in}(h)\\ h\in H\}}} \epsilon (h) (B_h B_{\bar{h}}+i_k j_k)\right) + \left(\sum\limits_{\substack{\{k=\text{in}(h)\\ h\in H\}}} \epsilon ^2(h) \xi B_h B_{h}^\d-\epsilon^2 (h) \xi B_{\bar{h}}^\d B_{\bar h}-\epsilon(h)\xi^2 B_{\bar{h}}^\d B_h^\d\right) +\\
+&\xi i_k i_k^\d -\xi j_k^\d j_k -\xi ^2 j_k^\d i_k^\d = \\
=& \mu _\C (q)+\left(\sum\limits_{\substack{\{k=\text{in}(h)\\ h\in H\}}} \xi (B_hB_h^\d-B_{\bar{h}}^\d B_{\bar{h}}+i_k i_k^\d -j_k^\d j_k )\right)- \xi^2\left(\sum\limits_{\substack{\{k=\text{in}(h)\\ h\in H\}}}\epsilon(h) B_{\bar{h}}^\d B_h^\d + j_k^\d i_k^\d\right) = \\
=& \frac{2\xi}{i}\mu_\R (q) -\xi ^2 \mu_{\C}(q)^{\dagger}.
 \end{split}   
\end{equation*}
As before, using the equations \eqref{initial} we obtain:
\begin{equation*}
    \mu_\C(q_\xi)= -2\xi i \mu_\R (q) -\xi ^2 \mu_\C (q)^{\d}=-2\xi i \mu_\R(q)= -2\xi i \zeta_\R.
\end{equation*}
\end{proof}

\subsection{Geometry of the $\C^{\star}$-action}

The space $\mathcal M_{\zeta_{\R}}$ caries a natural $S^1$-action given by 
\begin{equation}
(B_h,i_k,B_{\bar h}, j_k) \longmapsto (B_h,i_k,tB_{\bar h}, tj_k)
\end{equation}
for any $t\in S^1$. The $S^1$-action preserves the real moment map equation $\mu_{\R}(B,i,j) = \zeta_{\R}$ and scales the complex moment map equation $\mu_{\C}(t\cdot(B,i,j)) = t\mu_{\C}(B,i,j)$. Therefore the $S^1$-action sends points in $\mathfrak M_{((1-|\xi| ^2)\zeta_{\R},-2i\xi \zeta_{\R})}$ to points in $\mathfrak M_{((1-|\xi| ^2)\zeta_{\R},-2it\xi \zeta_{\R})}$ since $|\xi| = |t\xi|$. 

\begin{theorem}[\cite{Na1994}, Theorem $5.1.$]\label{CSA}
    The $S^1$-action on $\mathcal{M}_{\zeta_\R}$ has the following properties:
    \begin{enumerate}[label=(\roman*)]
    \item 
    The natural projection map $\pi: \mathcal M_{\zeta_\R}\longrightarrow   \C$ is equivariant, here $S^1$ acts on $\xi\in\C$ with weight 1,
    \item 
    It preserves the complex structure I and the metric,
    \item 
    The holomorphic symplectic form $\omega_\C$ transforms as   $\omega_\C\longrightarrow t\omega_\C$,
    \item 
    The corresponding moment map for the $S^1$-action
    \begin{equation*}
    F([B_h,i_k,B_{\bar h},j_k]) = \sum\limits_{ h \in \Omega}\|B_{\bar h}\|^2 + \sum\limits_{k=1}^n \|j_k\|^2
    \end{equation*}
    is proper,
    \item 
    The action is extended to a holomorphic $\C^*$-action. If we use the    holomorphic description            $\mathfrak{M}_{(\zeta_\R,\zeta_\C)}=H^s_{(\zeta_\R,\zeta_\C)}/G^\C_\mathbf{    v}$ then the $\C^*$-action is given by:
    \begin{equation*}
    G^\C_\mathbf{v}(B_h,i_k,B_{\bar h},j_k) \mapsto G^\C_\mathbf{v}(B_h,i_k,RB_{\bar h},Rj_k).
    \end{equation*}
    \end{enumerate}
\end{theorem}
The $\C^{\star}$-action defines a Bia\l{}ynicki-Birula type stratification of $\mathcal M_{\zeta_{\R}}$. It follows from the fourth part of Theorem \ref{CSA} that for any element $[(B_h,i_k,B_{\bar h},j_k)] \in \mathcal M_{\zeta_{\R}}$ the $\C^{\star}$-limit 
\begin{equation}
\lim\limits_{\xi\to 0}[(B_h,i_k,\xi B_{\bar h},\xi j_k)] = [(B_h^0,i_k^0,B_{\bar h}^0,j_k^0)]
\end{equation}
always exists and is contained in $\mathfrak M_{(\zeta_{\R},0)}$. Let $\mathfrak F$ denote the fixed point set of the $\C^{\star}$-action and write $\mathfrak F = \bigsqcup\limits_{a=1}^p \mathcal F_a$ for its decomposition into connected components. $\mathfrak F$ is a non-singular compact subvariety of $\mathbb M_{(\zeta_{\R},0)}$. 
\begin{definition} Given a fixed point component $\mathfrak F_a$ we define the \textit{attractor set} of $\mathfrak F_a$ as:
\begin{equation}
W_a := \{[(B_h,i_k,B_{\bar h},j_k)]\in \mathcal M_{\zeta_{\R}}~|~\lim\limits_{\xi\to 0}[(B_h,i_k,\xi B_{\bar h},\xi j_k)] \in \mathfrak F_a\}.
\end{equation}
\end{definition}
Each $W_a$ is connected and foliated by $\C^{\star}$-orbits $W_a(B_h^0,i_k^0,B_{\bar h}^0,j_k^0)$ defined as 
\begin{equation*}
W_a(B_h^0,i_k^0,B_{\bar h}^0,j_k^0):= \{[(B_h,i_k,B_{\bar h},j_k)]\in \mathcal M_{\zeta_{\R}}~|~\lim\limits_{\xi\to 0}[(B_h,i_k,\xi B_{\bar h},\xi j_k)]= [(B_h^0,i_k^0,B_{\bar h}^0,j_k^0)] \in \mathfrak F_a\}
\end{equation*}
and $\mathcal M_{\zeta_{\R}} = \bigsqcup\limits_{a} W_a$. The restrictions $W_a^0 := W_a \cap \pi^{-1}(0)$ and $W_a^1:= W_a\cap \pi^{-1}(1)$ define stratifications of the quiver variety $\mathfrak{M}_{(\zeta_\R,0)}$ and $\mathfrak{M}_{(0,-2i\zeta_\R)}$ respectively such that:
\begin{equation}
\begin{split}
        \mathfrak{M}_{(\zeta_\R,0)}&=\bigsqcup_a W_a^0 \\
    \mathfrak{M}_{(0,-2i\zeta_\R)}&=\bigsqcup_a W^1_a.
    \end{split}
\end{equation}
\begin{theorem}
The manifolds 
\begin{equation}
\begin{split}
W_a^0(B_h^0,i_k^0,B_{\bar h}^0,j_k^0) :=& W_a(B_h^0,i_k^0,B_{\bar h}^0,j_k^0)\cap \mathfrak M_{(\zeta_{\R},0)} \subset \mathfrak M_{(\zeta_{\R},0)}\\
W_a^1(B_h^0,i_k^0,B_{\bar h}^0,j_k^0) :=& W_a(B_h^0,i_k^0,B_{\bar h}^0,j_k^0)\cap \mathfrak M_{(0,-2i\zeta_{\R})} \subset \mathfrak M_{(0,-2i\zeta_{\R})}
\end{split}
\end{equation}
are holomorphic Lagrangian submanifolds of the respective spaces, diffeomorphic to $\C^N$. 
\end{theorem}
The proof of the theorem for $W_a^0(B_h^0,i_k^0,B_{\bar h}^0,j_k^0)$ will be given after the introduction of the Bialynicki-Birula slice and the second part will follow from the explicit parametrization of $W_a^1(B_h^0,i_k^0,B_{\bar h}^0,j_k^0)$ by the conformal limit map.

\section{The Bia\l{}ynicki-Birula slice}

The other main ingredient in the proof of existence of the conformal limit for Higgs bundles is the existence of the Bia\l{}ynicki-Birula (BB) slice, a finite dimensional subvariety of $\mathcal H^{ps}$ diffeomorphic to $W^0_a(\bar\pa_0,\Phi_0)$. In order to define the BB slice, we first define a different slice which we call the Hodge slice and then we restrict to a $\C^{\star}$-fixed point in order to define the BB slice. 

\subsection{The Hodge slice}

In order to define the Hodge slice we first need to understand the deformation complex at a stable point of a quiver variety. To do this we have to use the definition of the quiver variety through the GIT quotient construction. Recall that this is given by 
\begin{equation*}
\mathfrak M_{\zeta} \cong H_{\zeta}^s/ G_{\mathbf{v}}^{\C}. 
\end{equation*}
Therefore, at a stable point $p = (B_h,i_k,B_{\bar h},j_k)$ the deformation complex is given by 
\begin{equation}
\mathfrak g_{\mathbf{v}}\otimes \C \xlongrightarrow{l_p} \mathbb M \xlongrightarrow{ \,d\mu_{\C}(p,\star)} \mathfrak g_{\mathbf{v}}\otimes \C. 
\end{equation}
Here, $l_p$ is the linearization of the gauge group action 
\begin{equation}\label{gauge group}
    g\cdot(B_h,i_k,B_{\bar h},j_k)=(g_{in(h)}B_h g_{out(h)}^{-1},g_k i_kg_{out(h)B_{\bar h}}g_{in(h)}^{-1},j_kg_k^{-1}),
\end{equation}
and writing $g_k=e^{t\xi_k}$, it takes the form:
\begin{equation}
    l_p(\xi):=\frac{\,d}{\,dt}\Bigg|_{t=0} e^{t\xi}\cdot p = (\xi_{\text{in}(h)}B_h - B_h\xi_{\text{out}(h)}, \xi_ki_k, \xi_{\text{out}(h)}B_{\bar h} - B_{\bar h}\xi_{\text{in}(h)}, -j_k\xi_k).
\end{equation}
In order to obtain a representation of the tangent space at $[p]$ it is most convenient to work with the adjoint $l_p^{\star}$ of $l_p$. Since the explicit form of the adjoint will be essential later on we give a careful derivation of its particular form. The pairing of $l_p(\xi)$ with an element $q=(A_h,I_k,A_{\bar h},J_k)$ of the quiver representation space $\mathbb{M}$ is given by:
\begin{equation}\label{step1}
\langle l_p(\xi),q\rangle = \Tr\left( \sum\limits_{h\in \Omega} ((\xi_{\text{in}(h)}B_h - B_h\xi_{\text{out}(h)})A^{\dagger}_h+ (\xi_{\text{out}(h)}B_{\bar h} - B_{\bar h}\xi_{\text{in}(h)})A^{\dagger}_{\bar h}) +\sum\limits_k \xi_ki_kI_k^{\dagger} - J_k^{\dagger}j_k\xi_k \right).
\end{equation}
To calculate the adjoint $l_p^*(\xi)$ we expand both sides of the definition
\begin{equation}\label{adjoint}
    \langle l_p(\xi),q\rangle=\langle\xi,l_p^*(q)\rangle.
\end{equation}
Using equation \eqref{step1} we have:
\begin{equation}\label{step2}
    \begin{split}
        \langle l_p(\xi),q\rangle =& \Tr(\sum\limits_{h\in \Omega} (\xi_{in(h)}B_h-B_h\xi_{out(h)})A_h^\d )+ \\
        +& \Tr(\sum\limits_{h\in \Omega} (\xi_{out(h)} B_{\hb}-B_{\hb}\xi_{in(h)})A_{\hb})+\Tr(\sum\limits_{k}\xi_k i_kI_k^\d-J_k^\d j_k\xi_k)= \\
        =& \sum\limits_{h\in \Omega} [\Tr(\xi_{in(h)}B_hA_h^\d)-\Tr(B_h\xi_{out(h)}A_h^\d)+\Tr(\xi_{out(h)}B_{\hb}A_{\hb}^\d )-\\
        -& \Tr(B_{\hb}\xi_{in(h)} A_{\hb}^\d)]+\Tr(\sum\limits_{k}\xi_ki_kI_k^\d-J_k^\d j_k\xi_k).
    \end{split}
\end{equation}
Observing that 
\begin{equation*}
    \begin{split}
        \Tr(B_h\xi_{out(h)}A_h^\d)&=\Tr(\xi_{out(h)}A_h^\d B_h) \\
        \Tr(B_{\hb}\xi_{in(h)}A_{\hb}^\d)&=\Tr(\xi_{in(h)}A_{\hb}^\d B_{\hb})\\
        \Tr(J_k^{\dagger}j_k\xi_k) &= \Tr(\xi_kJ_k^{\dagger}j_k),
    \end{split}
\end{equation*}
 we can rewrite equation \eqref{step2} as:
\begin{equation*}
    \begin{split}
        \langle l_p(\xi),q \rangle =& \sum\limits_{h\in \Omega} \left(\Tr(\xi_{in(h)}B_hA_h^\d)-\Tr(\xi_{out(h)}A_h^\d B_h)\right)+\sum\limits_{h\in \Omega} \left(\Tr(\xi_{out(h)}B_{\hb}A_{\hb}^\d )- \Tr(\xi_{in(h)} A_{\hb}^\d B_{\hb})\right)+ \\
        +& \Tr\left(\sum\limits_{k}\xi_k(i_kI_k^\d-J_k^\d j_k)\right)= \\
        =& \sum\limits_{\substack{\{k=in(h)\\ h \in \Omega\}}} \Tr(\xi_k(B_hA_h^\d -A_{\hb}^\d B_{\hb }))+\sum\limits_{\substack{\{k=out(h)\\ h \in \Omega\}}} \Tr(\xi_k(B_{\hb}A_{\hb}^\d -A_h^\d B_h))+\\ 
         +& \Tr\left(\sum\limits_{k}\xi_k(i_kI_k^\d-J_k^\d j_k)\right)=\\
        =&\sum\limits_{k}   \Tr\left(\xi_k \left(\sum\limits_{\substack{\{k=in(h)\\ h \in \Omega\}}}B_h A_h^\d -A_{\hb}^\d B_{\hb}+\sum\limits_{\substack{\{k=out(h)\\ h \in \Omega\}}}B_{\hb}A_{\hb}^\d -A_h^\d B_h+i_k I_k^\d-J_k^\d j_k\right)\right) . 
        \end{split}
\end{equation*}
Define 
\begin{equation*}
M_k:=\sum\limits_{\substack{k=in(h)\\ h \in \Omega}}B_h A_h^\d -A_{\hb}^\d B_{\hb}+\sum\limits_{\substack{k=out(h)\\ h \in \Omega}}B_{\hb}A_{\hb}^\d -A_h^\d B_h+i_k I_k^\d-J_k^\d j_k.
\end{equation*}
Using this,  \eqref{adjoint} can be rewritten as:
\begin{equation*}
    \begin{split}
       \sum \limits_{k} \Tr\xi_kM_k &=\sum \limits_{k}\Tr \xi_k(l_p^*(q))_k^\d \Rightarrow \\
        M_k &=(l_p^*(q))_k^\d \Rightarrow \\
        l_p^*(q)=M_k^\d &=\sum\limits_{\substack{k=in(h)\\ h \in \Omega}}A_hB_h^\d-B_{\hb}A_{\hb}+\sum\limits_{\substack{k=out(h)\\ h \in \Omega}}A_{\hb}B_{\hb}^\d
        -B_h^\d A_h+I_ki_k^\d -j_k^\d J_k  =  \\
        &=\sum\limits_{\substack{k=in(h)\\ h \in \Omega}} A_hB_h^\d-B_{\hb}A_{\hb}+\sum\limits_{\substack{k=in(h)\\ h \in \bar \Omega}} A_hB_h^\d -B_{\hb}^\d A_{\hb}+I_ki_k^\d -j_k^\d J_k= \\
        &=\sum \limits_{\substack{h \in H \\ k=in(h)}}A_h B_h^\d-B_{\hb}^\d A_{\hb}+I_ki_k^\d -j_k^\d J_k.
\end{split}
\end{equation*}
We are now in position to define the tangent space of $\mathfrak M_{(\zeta_{\R},0)}$ at $[p]$. It is given by 
\begin{equation}
T_{[p]}\mathfrak M_{(\zeta_{R},0)} = \{ q\in \mathbb M~|~ \,d\mu_{\C}(p,q) =0 \text{ and } l^{\star}_p(q) =0\}. 
\end{equation}
Following the same intuition as in \cite{CW} we define the \emph{Hodge slice} by changing the linearized equation $\,d\mu_{\C}(p,q) =0$ with the original quadratic in $q$ complex moment map $\mu_{\C}(p+q) = 0$:
\begin{definition}
    Given  $p=(B_h,i_k,B_{\bar h },j_k)\in \mathfrak M_{(\zeta_{\R},0)}$ the Hodge slice for Nakajima quiver varieties is given by:
    \begin{equation}
        S_p=\{q=(A_h,I_k,A_{\bar h},J_k) \in \mathbb M~|~ \mu_\C(p+q)=0, l_p ^*(q)=0\}.
    \end{equation}
\end{definition}
\begin{proposition}\label{HLH}
The map $P$ defined by  $q \to [p+q]$ is a homeomorphism from a neighborhood of the origin in $S_p$ to a neighborhood of $[p]\in \mathfrak M_{(\zeta_{\R},0)}$. 
\end{proposition}
\begin{proof}
Consider the following map 
\begin{equation*}
P:S_p \to T_{[p]}\mathfrak M_{(\zeta_{R},0)}
\end{equation*}
given by 
\begin{equation*}
P(q ) = q + (\,d\mu_{\C})^{\star}G(-\,d\mu_{\C}(p,q))
\end{equation*}
where $G$ is the inverse of $\,d\mu_{\C}(p, *)\circ (\,d\mu_{\C}(p,*))^{\star}$ defined on the image $\mathrm{Im}(\,d\mu_{\C}(p,*))$. It is rather straightforward to show that it is invertible and that $\,d\mu(p,P(q)) = 0$. Therefore this map is a diffeomorphism from $S_p$ to $T_{[p]}\mathfrak M_{(\zeta_{R},0)}$ and composing it with the exponential map gives the required local diffeomorphism. 
\end{proof}

\subsection{Bia\l{}ynicki-Birula slice}
Finally, we are ready to define the Bia\l{}ynicki-Birula slice. Let $p_0=(B_h^0,i_k^0,B_{\bar h}^0,j_k^0)$ be an element of the quiver representation space $\mathbb M$, which is a representative of a $\C^*$-fixed point. This in particular implies that for all $t\in S^1$ there exists a gauge transformation $g_t\in G_{\mathbf{v}}$ such that 
\begin{equation}
g_t\cdot (B_h^0,i_k^0,tB_{\bar h}^0,tj_k^0) = (B_h^0,i_k^0,B_{\bar h}^0,j_k^0).
\end{equation}
Clearly $g_t$ depends on the fixed point but we choose to suppress this dependence in the notation. The family $g_t$ acts on the vector spaces $V_k$ and decomposes them into direct sums of eigenspaces
\begin{equation}
V_k = \bigoplus\limits_{m} V_k^m.
\end{equation}
This decomposition induces decompositions of $\mathbb{M}_{\Omega}$ and $\mathbb M_{\overline{\Omega}}$:
\begin{equation*}
    \mathbb{M}=\mathbb{M}_{\Omega}\oplus \mathbb{M}_{\overline{\Omega}} = \bigoplus\limits_{i=-a}^{b+1}\mathbb{M}_{\Omega}^i\oplus\bigoplus\limits_{j = -b}^{a+1}\mathbb{M}_{\overline{\Omega}}^j
\end{equation*}
for some positive integers $a,b$ where $S^1$ acts for example on $\mathbb{M}_{\Omega}^i$ as multiplication by $t^i$. The particular summation endpoints come from Theorem $3.4$ $(\text{iii})$ \ref{CSA} that the symplectic form transforms as  $\omega_{\C}\to t\omega_{\C}$ under the $S^1$-action and because the symplectic form is non-degenerate. $g_t$ also acts on $\mathfrak g_{\mathbf{v}}\otimes\C$ by conjugation giving a similar decomposition as above:
\begin{equation*}
\mathfrak g_{\mathbf{v}}\otimes\C = \bigoplus\limits_{-c}^c (\mathfrak g_{\mathbf{v}}\otimes\C)^c.
\end{equation*}
For all graded objects $G = \bigoplus\limits_{m} G^m$ let us fix the notation $ G^{(d)} := \bigoplus\limits_{d\le i} G^i$. Observe also that $p_0 \in \mathbb M_{\Omega}^0\oplus \mathbb M_{\overline{\Omega}}^{-1}$. Therefore for $p_0$ we get a refinement for the Hodge slice coming from the following diagram
\begin{equation}
(\mathfrak g_{\mathbf{v}}\otimes\C)^{(1)}\xlongrightarrow{l_{p_0}(*)}  \mathbb M_{\Omega}^{(1)}\oplus \mathbb M_{\overline{\Omega}}^{(0)}\xlongrightarrow{\mu_{\C}(P+*)}(\mathfrak g_{\mathbf{v}}\otimes\C)^{(0)}.
\end{equation}
\begin{definition}
We define the Bia\l{}ynicki-Birula slice 
\begin{equation}
S^+_{p_0}:= \{ q= (A_h,I_k,A_{\bar h},J_k)\in \mathbb M_{\Omega}^{(1)}\oplus \mathbb M_{\overline{\Omega}}^{(0)}~|~\mu_{\C}(p_0+q) =0 \text{ and } l^{\star}_{p_0}(q) = 0\}.
\end{equation}
\end{definition}
\begin{proposition}
$S^+_{p_0}$ is half the dimension of $\mathfrak M_{(\zeta_{\R},0)}$. 
\end{proposition}
\begin{proof}
This is a combination of the following observations. First, the Kuranishi map of Proposition \ref{HLH} restricted to $S^+_{p_0}$ provides an in fact global homeomorphism to 
\begin{equation}
T_{p_0}^+\mathfrak M_{(\zeta_{\R},0)} := \{ q= (A_h,I_k,A_{\bar h},J_k)\in \mathbb M_{\Omega}^{(1)}\oplus \mathbb M_{\overline{\Omega}}^{(0)}~|~\,d\mu_{\C}(p_0,q) =0 \text{ and } l^{\star}_{p_0}(q) = 0\}.
\end{equation}
The local part is trivial and the global follows from the fact that both spaces scale in the same way when acted by the $\C^{\star}$-action and the Kuranishi map is $\C^{\star}$-equivariant.  Second, stability of the point $p_0$ implies that the maps $\,d\mu_{\C}(p_0,q)$ and $l^{\star}_{p_0}(*)$ are surjective. For the second one this can be checked using the explicit formula we derived in the previous subsection and the first is similar. Third, by duality, $\dim((\mathfrak g_{\mathbf{v}}\otimes\C)^{(1)}) + \dim((\mathfrak g_{\mathbf{v}}\otimes\C)^{(0)}) = \dim(\mathfrak g_{\mathbf{v}}\otimes\C) = 2\dim G_{\mathbf{v}}$. Fourth, the dimension of $\mathbb M_{\Omega}^{(1)}\oplus \mathbb M_{\overline{\Omega}}^{(0)}$ is half that of $\mathbb M$. This follows since the former space is isotropic with respect to $\omega_{\C}$ and adding anything to it will by duality pair non-trivially with an element in the space.
\end{proof}
We close this subsection with a lemma explaining exactly how the $\C^{\star}$-action acts on the BB slice. This will be useful later on. 
\begin{lemma}
    Assume $(A_h,I_k,A_{\bar h},J_k) \in S_{p_0}^+$. Then $A_R:=g_R\cdot(A_h,I_k,RA_{\bar h},RJ_k) \in S_{p_0}^+.$
\end{lemma}
\begin{proof}
   We start from the complex moment map. Recall that $g_R\cdot (B_h^0,i_k^0,RB_{\bar h}^0,Rj_k^0)=(B_h^0,i_k^0,B_{\bar h}^0,j_k^0).$ Therefore,
   \begin{equation*}
       \begin{split}
           p_0+A_R&=(B_h^0,i_k^0,B_{\bar h}^0,j_k^0)+g_R\cdot (A_h,I_k,RA_{\bar h},RJ_k)= \\
           &=g_R\cdot (B_h^0,i_k^0,RB_{\bar h}^0,Rj_k^0)+g_R\cdot (A_h,I_k,RA_{\bar h},RJ_k)= \\
           &=g_R\cdot (B_h^0+A_h,i_k^0+I_k,R(B_{\bar h}^0+A_{\bar h}),R(j_k^0+J_k)).
       \end{split}
   \end{equation*}
   Since the complex moment map is invariant under complex gauge transformation and equivariant with respect to the $\C^{\star}$-action, we have that:
   \begin{equation*}
       \begin{split}
           \mu_\C(p_0+A_R)&=\mu_\C(g_R\cdot (B_h^0+A_h,i_k^0+I_k,R(B_{\bar h}^0+A_{\bar h}),R(j_k^0+J_k)))=\\
           &=\mu_\C(B_h^0+A_h,i_k^0+I_k,R(B_{\bar h}^0+A_{\bar h}),R(j_k^0+J_k))= \\
           &=R \mu_\C (B_h^0+A_h,i_k^0+I_k,B_{\bar h}^0+A_{\bar h},j_k^0+J_k)=0,
       \end{split}
   \end{equation*}
   since $(A_h,I_k,A_{\bar h},J_k) \in S_{p_0}^+$. Therefore $\mu_\C(p_0+A_R)=0.$

   We proceed with the real moment map. We want to prove that:
   \begin{equation*}
       \begin{split}
           \sum \limits_{\substack{\{h \in H\\k=in(h)\}}} & g_R^{in(h)}A_h(g_R^{out(h)})^{-1}(B_h^0)^\dagger -(B_{\bar h}^0)^\dagger(Rg_R^{out(h)}A_{\bar h}(g_R^{in(h)})^{-1}+ \\
           &+g_R^kI_k(i_k^0)^\dagger-(j_k^0)^\dagger RJ_k(g_R^k)^{-1}=0
       \end{split}
   \end{equation*}
   Using the identification $g_R^{in(h)}=g_R^k$ we can write:
   \begin{equation}\label{sum}
       \begin{split}
           \sum \limits_{\substack{\{h \in H\\k=in(h)\}}}& g_R^{in(h)}A_h(g_R^{out(h)})^{-1}(B_h^0)^\dagger -(B_{\bar h}^0)^\dagger(Rg_R^{out(h)}A_{\bar h}(g_R^{in(h)})^{-1}+ \\
           +&g_R^kI_k(i_k^0)^\dagger-(j_k^0)^\dagger RJ_k(g_R^k)^{-1}= \\
           =&g_R^k\bigg(\sum \limits_{\substack{\{h \in H\\k=in(h)\}}} A_h (g_R^{out(h)})^{-1}(B_h^0)^\dagger g_R^{in(h)}-R(g_R^{in(h)})^{-1}(B_{\bar h}^0)^\dagger g_R^{out(h)}A_{\bar h}+ \\
           +&I_k (i_k^0)^\dagger g_R^k-R(g_R^k)^{-1}(j_k^0)^\dagger J_k\bigg )(g_R^k)^{-1}.
       \end{split}
   \end{equation}
   But now we can use the following equalities:
   \begin{equation}\label{i}
   (g_R^{out(h)})^{-1}(B_h^0)^\dagger g_R^{in(h)}=\big(g_R^{in(h)}(B_h^0)^\dagger (g_R^{out(h)})^{-1}\big)^\dagger =(B_h^0)^\dagger,
   \end{equation}
  \begin{equation}\label{ii}
  \begin{split}
  R(g_R^{in(h)})^{-1}(B_{\bar h}^0)^\dagger g_R^{out(h)}&=R(g_R^{out(h)}B_{\bar h}^0(g_R^{in(h)})^{-1})^\dagger=R(R^{-1}B_{\bar h}^0)^\dagger=(B_{\bar h}^0)^\dagger,
  \end{split}
  \end{equation}
  \begin{equation}\label{iii}
      \begin{split}
          (i_k^0)^\dagger g_R^k=(g_R^ki_k^0)^\dagger=(i_k^0)^\dagger,
      \end{split}
  \end{equation}
  \begin{equation}\label{iiii}
      \begin{split}
      R(g_R^k)^{-1}(j_k^0)^\dagger =R(j_k^0(g_R^k)^{-1})^\dagger=R(R^{-1}j_k^0)^\dagger=(j_k^0)^\dagger.
      \end{split}
  \end{equation}
  Plugging the equations \eqref{i}, \eqref{ii}, \eqref{iii} and \eqref{iiii} in \eqref{sum} we obtain:
\begin{equation*}
    \sum\limits_{\substack{h\in H\\k=in(h)}}A_h(B_h^0)^\dagger-(B_{\bar h}^0)^\dagger A_{\bar h}+I_k(i_k^0)^\dagger-(j_k^0)^\dagger J_k=l_{p_0}^*(A_h,I_k,A_{\bar h},J_k)=0
\end{equation*}
   since $(A_h,I_k,A_{\bar h},J_k)\in S_{p_0}^+.$
\end{proof}

\subsection{A technical lemma}\label{IFT}

The aim of this subsection is to show that the BB slice $S^+_{p_0}$ and the holomorphic Lagrangian $W_a^0(p_0)$ can be related by a complex gauge transformation and obtain a quantitative estimate on the form of this transformation for points $A_R$ with $R$ small. 

Let us start with a comment. While the group action by $G_{\mathbf{v}}$ preserves both the real and the complex moment map i.e.
\begin{equation}
\begin{split}
\mu_{\R}(g\cdot(B_h,i_k,B_{\bar h}, j_k)) &= \zeta_{\R}\\
\mu_{\C}(g\cdot(B_h,i_k,B_{\bar h}, j_k)) &= \zeta_{\C},
\end{split}
\end{equation}
For every $g\in G_{\mathbf{v}}$, the group action by $G_{\mathbf{v}}^{\C}$ only preserves the complex moment map. Therefore when we want to find a solution to both moment maps it is customary to start with a pair satisfying the complex moment map and then search for a complex gauge transformation $g$ for which 
\begin{equation*}
\mu_{\R}(g\cdot(B_h,i_k,B_{\bar h}, j_k) )= \zeta_{\R}.
\end{equation*}
This observation has been explored to great length in order to find solutions to the so called equations carrying a Hermitian Yang--Mills structure. It is also the underlying reason why the Hodge and BB slices are comprized of points which already satisfy the complex moment map. 
Before we state the main lemma of this subsection, we show how, assuming we have a point which satisfies the complex moment map and almost satisfies the real moment map, to apply the inverse function theorem in order to get an exact solution.  Specifically, let $p = (B_h,i_k,B_{\bar h}, j_k)\in\mathbb M$ such that  
\begin{equation}
\begin{split}
\mu_{\C}(B_h,i_k,B_{\bar h}, j_k) &= \zeta_{\C}\\
\mu_{\R}(B_h,i_k,B_{\bar h}, j_k) &= \zeta_{\R} +\text{small error}.
\end{split}
\end{equation}
The goal is to find a complex gauge transformation $g\in G_{\mathbf{v}}^{ \C}$ such that 
\begin{equation}
\mu_{\R}(g\cdot(B_h,i_k,B_{\bar h}, j_k)) = \zeta_{\R}.
\end{equation}
More explicitly: 
\begin{equation}
\begin{split}
\frac{i}{2}&\sum\limits_{\substack{\{h\in H\\k=in(h)\}}} g_{in(h)}B_hg_{out(h)}^{-1}(g_{out(h)}^{-1})^{\dagger}B_h^{\dagger}g_{in(h)}^{\dagger} - (g_{in(h)}^{-1})^{\dagger}B_{\bar h}^{\dagger}g_{out(h)}^{\dagger}g_{out(h)}B_{\bar h}g_{in(h)}^{-1} +\\ &g_ki_ki_k^{\dagger}g_{k}^{\dagger} - (g_k^{-1})^{\dagger}j_k^{\dagger}j_kg_k^{-1} = \zeta_{\R}^k.
\end{split}
\end{equation}
If we conjugate this equality with $g_{in(h)} = g_k$ then since $\zeta_{\R}$ is in the center, $g_k^{-1} \zeta_{\R}^k g_k = \zeta_{\R}$ and the above equality can be rewritten as 
\begin{equation}
\begin{split}
\frac{i}{2}&\sum\limits_{\substack{\{h\in H\\k=in(h)\}}} B_hg_{out(h)}^{-1}(g_{out(h)}^{-1})^{\dagger}B_h^{\dagger}g_{in(h)}^{\dagger}g_{in(h)} - g_{in(h)}^{-1}(g_{in(h)}^{-1})^{\dagger}B_{\bar h}^{\dagger}g_{out(h)}^{\dagger}g_{out(h)}B_{\bar h} + \\&i_ki_k^{\dagger}g_{k}^{\dagger}g_k - g_k^{-1}(g_k^{-1})^{\dagger}j_k^{\dagger}j_k = \zeta_{\R}^k.
\end{split}
\end{equation}
Using the notation $h:=(g_1^\dagger g_1,...,g_n^\dagger g_n)=(h_1,...,h_n)$ we can further rewrite the above equation as 
\begin{equation}
\frac{i}{2}\sum\limits_{\substack{\{h\in H\\k=in(h)\}}} B_hh_{out(h)}^{-1}B_h^{\dagger}h_{in(h)} - h_{in(h)}^{-1}B_{\bar h}^{\dagger}h_{out(h)}B_{\bar h} + i_ki_k^{\dagger}h_k - h_k^{-1}j_k^{\dagger}j_k = \zeta_{\R}^k.
\end{equation}
Now, $h_k$ are hermitian matrices and as such, we can write them as $h_k = e^{2\xi_k}$ where $\xi_k \in i \mathfrak g_k$ so that $\xi = (\xi_1,...,\xi_n) \in i \mathfrak g_{\mathbf{v}}$. Therefore, we are looking for a $\xi$ which solves the equation
\begin{equation}\label{nonlinear}
\mu_{\R}(e^{\xi}\cdot (B_h,i_k,B_{\bar h}, j_k)) = \zeta_{\R}. 
\end{equation}
To find such $\xi$ we want to apply the inverse function theorem to the non-linear map 
\begin{equation}
i\mathfrak g_{\mathbf{v}}\ni\xi \xrightarrow{F} -2i\mu_{\R}(e^{\xi}\cdot (B_h,i_k,B_{\bar h}, j_k)) \in i\mathfrak g_{\mathbf{v}}.
\end{equation}
We have added the $-2i$ to cancel the $i/2$ in front of the real moment map so as to make the map go into $i\mathfrak g_{\mathbf{v}}$. Assuming that the original error is sufficiently small, what we need is to show that the linearization of the map $F$ at $\xi =0$ is bijective.   The linearization of $F$ has the following explicit form:
\begin{equation}
\mathcal L((B_h,i_k,B_{\bar h}, j_k),\xi)_k = \sum\limits_{\substack{\{h\in H\\k=in(h)\}}} B_h(B_h^{\dagger}\xi_{in(h)} - \xi_{out(h)}B_h^{\dagger}) - (B_{\bar h}^{\dagger}\xi_{out(h)} - \xi_{in(h)}B_{\bar h}^{\dagger})B_{\bar h} + i_ki_k^{\dagger}\xi_k + \xi_kj_k^{\dagger}j_k. 
\end{equation}
It is easy to check that this map is self-adjoint which means that in order to show that it is bijective it is enough to show that it is injective. Assume that it is not. This means that there exists a $\xi$ such that $\mathcal L((B_h,i_k,B_{\bar h}, j_k),\xi)_k = 0$, for all $k$. We take the inner product of this quantity against $\xi$ itself: 
\begin{equation}
\langle \mathcal L((B_h,i_k,B_{\bar h}, j_k),\xi), \xi\rangle = \Tr \sum\limits_{k}(\mathcal L((B_h,i_k,B_{\bar h}, j_k),\xi)^{\dagger}_k\xi_k\rangle.
\end{equation}
If we can show this is positive then we get a contradiction. Let us make an observation. 
\begin{equation}
\mathcal L((B_h,i_k,B_{\bar h}, j_k),\xi)^{\dagger}_k = \sum\limits_{\substack{\{h\in H\\k=in(h)\}}} (\xi_{in(h)}B_h - B_h\xi_{out(h)})B_h^{\dagger} - B_{\bar h}^{\dagger}(\xi_{out(h)}B_{\bar h} - B_{\bar h}\xi_{in(h)}) + \xi_ki_ki_k^{\dagger} + j_k^{\dagger}j_k\xi_k
\end{equation}
which from Section $3$ is exactly $l^{\star}_p(l_p(\xi))$ with $p = (B_h,i_k,B_{\bar h}, j_k)$. Therefore, 
\begin{equation}
\begin{split}
\langle \mathcal L((B_h,i_k,B_{\bar h}, j_k),\xi), \xi\rangle &= \langle l^{\star}_p(l_p(\xi)),\xi\rangle\\
&= \langle l_p(\xi),l_p(\xi)\rangle\\
&= \|l_p(\xi)\|^2. 
\end{split}
\end{equation}
Therefore $\mathcal L((B_h,i_k,B_{\bar h}, j_k),\xi)_k = 0$ implies $l_p(\xi)=0$. But we have assumed that $\zeta_{\R}$ is generic which  implies that the map $l_p(*)$ is injective and therefore $\xi $ must be zero. Therefore we have arrived at a contradiction so the map is indeed bijective and therefore if the error is small enough, the inverse function theorem immediately produces a small $\xi$ for which \eqref{nonlinear} holds.  

While this is the main idea of how to construct solutions, in order to prove that the conformal limit converges, we need to prove a more quantitative estimate for $\xi$ in terms of the error in our set up. First we fix some notation. Let 
\begin{equation}
\End(V_k) = \bigoplus\limits_{m} \End(V_k)_m
\end{equation}
where 
\begin{equation}
\End(V_k)_m = \bigoplus\limits_j \Hom (V_k^j, V_k^{j+m}). 
\end{equation}
Denote the set of functions which take values in $\End(V_k)_m$ and are of order $\mathcal O(R^j)$ as $\End(V_k)_m(R^j)$. 
\begin{lemma}\label{metric}
Let $A= (A_h,I_k,A_{\bar h}, J_k) \in S_{p_0}^+$ and write $(B_h,i_k,B_{\bar h},j_k) = p_0 +A_R$. Then there exists $R_A>0$ sufficiently small such that for all $R\le R_A$ there exists a $\xi_R \in i\mathfrak g_{\mathbf{v}}$ which satisfies the real moment map equation 
\begin{equation}
\mu_{\R}(e^{\xi_R}\cdot(p_0+A_R))_k = \zeta_{\R}^k
\end{equation}
and 
\begin{equation}\label{metricansatz}
(\xi_R)_k \in \End(V_k)_0(R^2) + \bigoplus\limits_{|m|= 1}\End(V_k)(R^3)+ ... + \bigoplus\limits_{|m|= M} \End(V_k)_m(R^{M+2})
\end{equation}
with $M$ being the maximal positive integer for which $\End(V_k)_M \neq \emptyset$.
\end{lemma}
\begin{proof}
We use the ansatz 
\begin{equation}
\xi_R = R^2\xi_0 + ... +R^{M+2}\xi_{M}
\end{equation}
with $\xi_j \in \bigoplus\limits_{|m|\le j}\End(V_k)_m$. We start by plugging in $p_0+A_R$ into the real moment map. The error that this will produce is of the form 
\begin{equation}\label{firsterror}
\begin{split}
\mu_{\R}(p_0+A_R) &= \zeta_{\R} + \End(V_k)_0(R^2) + \bigoplus\limits_{|m|= 1}\End(V_k)(R^3)+ ... + \bigoplus\limits_{|m|= M} \End(V_k)_m(R^{M+2})\\
&= E_0^{(1)}+E_1^{(1)}+...+ E_M^{(1)}.
\end{split}
\end{equation}
The reason there is no linear error term in $R$ is because the linear term turns out to be exactly equal to $l_{p_0}^{\star}(A_R) + l_{p_0}^{\star}(A_R)^{\dagger}$. This term is zero since $A$ and thus also  $A_R\in S^+_{p_0}$ and therefore $l_{p_0}^{\star}(A_R) =0$ by definition. Now, the linearization of the real moment map at $p_0+A_R$ satisfies 
\begin{equation}
\begin{split}
\mathcal L(p_0+A_R,\star) = &\mathcal L(p_0,\star ) +\\+& \|\star\|\left(\bigoplus\limits_{|m|\le 1}\End(V_k)(R)+ ... + \bigoplus\limits_{|m|= M} \End(V_k)_m(R^{M})\right).
\end{split}
\end{equation}
Now comes a crucial observation. Notice that 
\begin{equation}\label{gradingpreserving}
\mathcal L(p_0,\star): \bigoplus\limits_{|m|\le j}\End(V_k)_m \to \bigoplus\limits_{|m|\le j}\End(V_k)_m
\end{equation}
due to the particular form of $p_0$. Also since $p_0$ is stable, $\mathcal L(p_0,\star)$ is always invertible. Therefore we can find a unique $\xi_0$ to solve away the $E_0^{(1)}$ error term in \eqref{firsterror}. This means that we find a solution $\mathcal L(p_0,R^2\xi_0) = E_0^{(1)}$. This is the first step. Now, plug into the real moment map the point $e^{R^2\xi_0}\cdot(p_0+A_R)$. The error produced is going to be of the form 
\begin{equation}\label{seconderror}
\begin{split}
\mu_{\R}(e^{R^2\xi_0}\cdot(p_0+A_R)) &= \zeta_{\R} + \bigoplus\limits_{|m|\le 1}\End(V_k)(R^3)+ ... + \bigoplus\limits_{|m|= M} \End(V_k)_m(R^{M+2})\\
&= E_1^{(2)}+...+ E_M^{(2)}.
\end{split}
\end{equation}
The linearization of the moment map satisfies 
\begin{equation}
\begin{split}
\mathcal L(e^{R^2\xi_0}\cdot(p_0+A_R),\star) = &\mathcal L(p_0,\Id,\star ) +\\
+&\|\star\|\left(\bigoplus\limits_{|m|\le 1}\End(V_k)(R)+ ... + \bigoplus\limits_{|m|= M} \End(V_k)_m(R^{M})\right).
\end{split}
\end{equation}
Thanks to \eqref{gradingpreserving} we can solve away the error $E_1^{(2)}$ meaning finding $\xi_1\in  \bigoplus\limits_{|m|\le1}\End(V_k)$ such that $\mathcal L(p_0,R^3\xi_1) = E_1^{(2)}$. This is the second step. Repeating the above procedure we iteratively generate $\xi_j$ for $j$ up to $M-1$. Then it holds that for 
$e^{R^{M+1}\xi_{M-1}}\cdot ... \cdot e^{R^2\xi_0}\cdot (p_0+A_R)$ the error from the real moment map will be
\begin{equation}\label{finalerror}
\mu_{\R}(e^{R^{M+1}\xi_{M-1}}\cdot...\cdot e^{R^2\xi_0}\cdot (p_0+A_R)) = \bigoplus\limits_{|m|\le M}\End(V_k)_m(R^{M+2}).
\end{equation}
At this point we choose $R$ to be appropriately small so that we can just apply the inverse function theorem since the correction produced will be of the order $ \mathcal O(R^{M+2})$ and therefore will be the last piece from the ansatz $R^{M+2}\xi_M$.
\end{proof}
\begin{corollary}\label{GG}
For any $A= (A_h, I_k,A_{\bar h}, J_k)\in S^+_{p_0}$ there exists a unique gauge transformation $g_A$ such that $g_A\cdot(p_0+A) \in W_a^0(p_0)$. Conversely, for any point in $W_a^0(p_0)$ there exists a unique gauge transformation which takes it to an element $p_0+A$ with $A\in S^+_{p_0}$. 
\end{corollary}
\begin{proof}
We already know this when the norm of $A$ is small by Lemma \ref{metric} above. Therefore assume that we know that $[p_0+A]\in W_a^0(p_0)$. Now by Lemma \ref{CSA} and the definition of $W_a^0(p_0)$ we know that this implies that $(B_h^0+A_h,i_k^0+I_k,R(B_{\bar h}^0+A_{\bar h}), R(j_k^0+J_k))$ is also in the complex gauge orbit of a point in $W_a^0(p_0)$. This implies that there exists a unique $g^1_A$ such that 
\begin{equation*}
\mu_{\R}(g^1_A\cdot(B_h^0+A_h,i_k^0+I_k,R(B_{\bar h}^0+A_{\bar h}), R(j_k^0+J_k))) = \zeta_{\R}.
\end{equation*}
But then for $g_A = g^1_Ag^{-1}_R$ it holds that 
\begin{equation*}
\begin{split}
\zeta_{\R} &= \mu_{\R}(g^1_A\cdot(B_h^0+A_h,i_k^0+I_k,R(B_{\bar h}^0+A_{\bar h}), R(j_k^0+J_k))) \\
&= \mu_{\R}(g_A\cdot(p_0+A_R)).
\end{split}
\end{equation*}
Therefore, once the theorem hods for small $A\in S^+_{p_0}$ we can use the $\C^{\star}$-action to scale it to any $A$. 
\end{proof}

\section{Conformal limit for quiver varieties}

We are finally in position to prove that the conformal limit exists. We start with the following lemma. 
\begin{lemma}\label{PA}
    Let $p_0=(B_h^0,i_k^0,B_{\bar{h}}^0,j_k^0)$ a fixed point and $A=(A_h,I_k,A_{\bar{h}},J_k)$ an element of $S_{p_0} ^+$. Then the element $p_A=(B_h^0+A_h-\hbar(B_{\bar{h}}^0)^\d,i_k^0+I_k-{\hbar}(j_k^0)^\d,{\hbar}^{-1}(B_{\bar{h}}^0+A_{\bar{h}})+(B_h^0)^\d,\hbar^{-1}(j_k^0+J_k)+(i_k^0)^\d)$ of the quiver representation space $\mathbb{M}$ satisfies the equation:
    \begin{equation}
        \mu_\C(p_A)=-2i\zeta_\R.
    \end{equation}
\end{lemma}
\begin{proof}
    Recall that $A \in S_{p_0}^+$, is equivalent to the conditions: 
    \begin{equation} \label{1}
    \begin{split}
&l_{p_0}^*(A)=
\sum\limits_{\substack{\{h \in H\\k=in(h)\}}}A_h (B_h^0)^\d -(B_{\bar{h}} ^0)A_{\bar{h}}+I_k(i_k^0)^\d -(j_k^0)^\d J_k=0  
    \end{split}
    \end{equation}
and 
    \begin{equation}\label{2}
        \begin{split}
            \mu_\C(p_0+A)_k&=\sum\limits_{\substack{\{h \in H\\k=in(h)\}}} \epsilon(h)(B_h^0+A_h)(B_{\hb}^0 +A_{\hb})+(i_k^0+I_k)(j_k^0+J_k)   \\
            &=\sum\limits_{\substack{\{h \in H\\k=in(h)\}}}\epsilon(h)(B_h^0 B_{\hb}^0+B_h^0A_{\hb} +A_hB_{\hb}^0+A_hA_{\hb})+i_k^0 j_k^0 +i_k^0 J_k+I_k j_k^0 +I_kJ_k= 0.
        \end{split}
    \end{equation}
Using these identities we calculate the complex moment map:
\begin{equation*}
\begin{split}
    \mu_\C(p_0)_k=&\sum\limits_{\substack{\{h\in H\\k=in(h)\}}} \epsilon(h) (B_h^0+A_h-\hbar (B_{\hb}^0)^\d )(\hbar ^{-1}(B_{\hb}^0+A_{\hb})+(B_h^0)^\d)+\\
    &+(i_k^0 +I_k-\hbar (j_k^0)^\d )(\hbar ^{-1}(j_k^0+J_k)+(i_k^0)^\d)= \\
    =&\sum\limits_{\substack{\{h\in H\\k=in(h)\}}} \epsilon(h)(B_h^0+A_h-\hbar (B_{\hb}^0)^\d )(\hbar ^{-1} B_{\hb}^0+\hbar ^{-1} A_{\hb} +(B_h^0)^\d )+\\
    &+(i_k^0 +I_k-\hbar (j_k^0)^\d )(\hbar^{-1} j_k^0 +\hbar ^{-1}J_k+(i_k^0)^\d)= \\
    =&\sum\limits_{\substack{\{h\in H\\k=in(h)\}}} \epsilon(h)(\hbar ^{-1} B_h^0 B_{\hb}^0+\hbar ^{-1} B_h^0A_{\hb}+B_h^0(B_h^0)^\d +\hbar^{-1} A_hB_{\hb}^0+\hbar ^{-1} A_hA_{\hb}+A_h(B_h^0)^\d- \\
    &- \hbar \hbar^{-1}(B_{\hb}^0)^\d B_{\hb}^0-\hbar \hbar^{-1}(B_{\hb}^0)^\d A_{\hb}-\hbar (B_{\hb}^0)^\d (B_h^0)^\d) +\hbar ^{-1}i_k^0 j_k^0+\hbar^{-1} i_k^0 J_k +i_k^0 (i_k^0)^\d +\\
    &+ \hbar^{-1}I_kj_k^0+\hbar^{-1}I_kJ_k +I_k(i_k^0)^\d -\hbar \hbar^{-1}(j_k^0)^\d j_k^0 -\hbar \hbar^{-1} (j_k^0)^\d J_k= \\
    =&\sum\limits_{\substack{\{h\in H\\k=in(h)\}}} \epsilon(h)(A_h (B_h^0)^\d -(B_{\hb}^0)^\d A_{\hb}+I_k(i_k^0)^\d -(j_k^0)^\d J_k+ \\
    &+ \hbar^{-1} \sum\limits_{\substack{\{h\in H\\k=in(h)\}}} \epsilon(h) (B_h^0B_{\hb}^o+ B_h^0A_{\hb}+A_h B_{\hb}^0+A_h A_{\hb})+i_k^0j_k^0+i_k^0J_k+I_kj_k^0+I_kJ_k+ \\
    &+\sum\limits_{\substack{\{h\in H\\k=in(h)\}}}\epsilon(h)(B_h^0(B_h^0)^\d -(B_{\hb}^0)^\d B_{\hb}^0 +i_k^0(i_k^0)^\d -(j_k^0)^\d j_k^0) +\\
    &+\sum\limits_{\substack{\{h\in H\\k=in(h)\}}} \epsilon(h)(-\hbar (B_{\hb}^0)^\d (B_h^0)^\d)-\hbar (j_k^0)^\d(i_k^0)^\d) = \\
    =& l_{p_0}^*(A)+\hbar ^{-1}\mu_\C(p_0+A)+\frac{2}{i}\mu_R(p_0)-\hbar \mu_\C(p_0^\d).
\end{split}
\end{equation*}
Using the equations \eqref{1},\eqref{2} and the fact that $p_0$ satisfies the equations $\mu _\R(p_0)=\zeta_{\R}$ and $\mu_\C(p_0)=0$) we obtain:
\begin{equation*}
    \mu_\C(p_A)=-2i\mu_\R(p_0)=-2i\zeta_\R.
\end{equation*}
\end{proof}
It follows from Lemma \ref{PA} and the genericity of $\zeta_{\R}$ that there exists a complex gauge transformation $g_{A,\hbar}$ such that 
$g_{A,\hbar}\cdot p_A$ satisfies the real moment map equation: 
\begin{equation}\label{pa}
    \mu_{\R}(g_{A,\hbar}\cdot p_A) = 0
\end{equation}
and therefore $g_{A,\hbar}\cdot p_A\in \mu^{-1}_{\R}(0)\cap\mu^{-1}_{\C}(-2i\zeta_{\R})$. 
Now, suppose we start with a point $p\in W_a^0(p_0)$. From Lemma \ref{GG} we know that there exists a unique gauge transformation such that $g_A^{-1}\cdot p = p_0+A$ with $A\in S^+_{p_0}$. Therefore, we might as well start with a $p_0+A$ with $A\in S^+_{p_0}$ and consider the conformal limit of $g_A\cdot(p_0+A)$ which we denote as $\mathcal {CL}_{\hbar}(p_0+A)$. 

Recall from \eqref{snake2} that the first step in defining the family of points whose limit we seek to show exists is to act by the $\C^{\star}$-action through multiplication by $R$. Since we can already assume $R$ to be appropriately small, we know from Lemma \ref{metric} that there exists a gauge transformation $g_{A,R}$ such that 
\begin{equation}
\begin{split}
h_{A,R}:=& g_{A,R}^{\dagger}g_{A,R}\\
=& \Id + \End(V_k)_0(R^2) + \bigoplus\limits_{|m|= 1}\End(V_k)(R^3)+ ... + \bigoplus\limits_{|m|= M} \End(V_k)_m(R^{M+2})
\end{split}
\end{equation}
and which satisfies
\begin{equation}
\begin{split}
\mu_{\R}(g_{A,R}\cdot(p_0+A_R)) &= \zeta_{\R}\\
\mu_{\C}(g_{A,R}\cdot(p_0+A_R)) &= 0.
\end{split}
\end{equation}
Next, after considering the twistor line for $\xi$ and acting through the $\C^{\star}$-action by $\xi^{-1}$ we get the family of points 
\begin{equation}
\begin{split}
P_R := (&G_{A,R}^{in(h)}B_h(G_{A,R}^{out(h)})^{-1} - \hbar R^2 (G_{A,R}^{in(h),\dagger})^{-1}B_{\bar h}^{\dagger}(G_{A,R}^{out(h)})^{\dagger}, G_{A,R}^k i_k - \hbar R^2((G_{A,R}^k)^{-1})^{\dagger}j_k^{\dagger}, \\ &\hbar^{-1}G_{A,R}^{out(h)}B_{\bar h}(G_{A,R}^{in(h)})^{-1} + ((G_{A,R}^{out(h)})^{-1})^{\dagger}B_h^{\dagger}(G_{A,R}^{in(h)})^{\dagger}, \hbar^{-1}j_k(G_{A,R}^k)^{-1} + i_k^{\dagger}(G_{A,R}^k)^{\dagger})
\end{split}
\end{equation}
where we have defined $G_{A,R} = g_{A,R}\cdot g_R$ so that $g_{A,R}\cdot (p_0+A_R) = G_{A,R}\cdot (B_h,i_k,RB_{\bar h},Rj_k)$.
This family of points satisfies the complex moment map equation $\mu_{\C}(P_R) = -2i\zeta_{\R}$ by construction. However, it is quite far from satisfying the real moment map equation and therefore we need to modify it before there is any hope of applying the inverse function theorem. Thus we consider the modified family of points 
\begin{equation}
\begin{split}
F_R:= g_{A,\hbar}\cdot G_{A,R}^{-1}\cdot P_R = g_{A,\hbar}\cdot(&B_h - \hbar R^2H_{A,R,in(h)}^{-1}B_{\bar h}^{\dagger}H_{A,R,out(h)}, ~i_k - \hbar R^2H_{A,R,k}^{-1}j_k^{\dagger},\\ &\hbar^{-1}B_{\bar h} + H_{A,R,out(h)}^{-1}B_h^{\dagger}H_{A,R,in(h)}, ~\hbar^{-1}j_k + i_k^{\dagger}H_{A,R,k}),
\end{split}
\end{equation}
where $H_{A,R} = g_Rh_{A,R}g_R$. We claim that the limit 
\begin{equation}
\begin{split}
\lim\limits_{R\to 0} F_R = &g_{A,\hbar}\cdot p_A = \\&g_{A,\hbar}\cdot(B_h^0+A_h-\hbar(B_{\bar{h}}^0)^\d,i_k^0+I_k-{\hbar}(j_k^0)^\d,\\
&~~~~~~~~~~~{\hbar}^{-1}(B_{\bar{h}}^0+A_{\bar{h}})+(B_h^0)^\d,\hbar^{-1}(j_k^0+J_k)+(i_k^0)^\d),
\end{split}
\end{equation}
and therefore for $R$ appropriately small, we will be able to apply the inverse function theorem to correct $F_R$ to a point $\widetilde{F_R}$ which actually satisfies both moment map equations and such that 
\begin{equation*}
\lim\limits_{R\to 0}\widetilde{F_R} = g_{A,\hbar}\cdot p_A.
\end{equation*}

That the limit is of the required form follows from the particularly nice properties of the metric $h_{A,R}$ which was constructed in Lemma \ref{metric}. Using these properties it is straightforward to check the above claim. Let us do this for the first term since the rest can be done in essentially the same way. What we show is that 
$H_{A,R,out(h)}B_{\bar h} H_{A,R,in(h)}^{-1} = R^{-2}B^0_{\bar h} + \mathcal O(1)$. In order to do this we have to analyze how 
\begin{equation}
H_{A,R,out(h)}B_{\bar h} H_{A,R,in(h)}^{-1}  =g_R^{out(h)}h_{A,R}^{out(h)}g_R^{out(h)}B_{\bar h}(g_R^{in(h)})^{-1}(h_{A,R}^{in(h)})^{-1}(g_R^{in(h)})^{-1}
\end{equation}
acts on an element $v\in V_{in(h)}^j$. Clearly $h_{A,R}^{-1}$ satisfies the same nice bounds \eqref{metricansatz} as $h_{A,R}$. We know that $ (g_R^{in(h)})^{-1}$ sends $v$ to $R^{-j}v$. We know that the part of $(h_{A,R}^{in(h)})^{-1}$ which sends $V_{in(h)}^j \to V_{in(h)}^{j\pm n}$ for some positive $n$ sends $R^{-j}v$ to $R^{-j + n+2}v_1 \in V_{in(h)}^{j\pm n}$ with $v_1$ bounded. Then $(g_R^{in(h)})^{-1}$ sends $R^{-j + n+2}v_1$ to $R^{-j + n+2-(j\pm n)}v_1 = R^{-2j + n+2\mp n}v_1$. Now, we know that if $B_{\bar h}$ sends $V_{in(h)}^a \to V_{out(h)}^b$ then $a\le b+1$. Therefore, it is enough to study the worst case which is when $b=a-1$. Otherwise the error will be decaying faster in $R$. So assume $B_{\bar h}$ sends $R^{-2j + n+2\mp n}v_1$ to $R^{-2j + n+2\mp n}v_2 \in V_{out(h)}^{j\pm n-1}$ with $v_2$ bounded. We know that $g_R^{out(h)}$ sends $R^{-2j + n+2\mp n}v_2$ to $R^{-2j + n+2\mp n +(j\pm n-1)}v_2$. We know that the part of $h_{A,R}^{out(h)}$ which sends $V_{out(h)}^{j\pm n-1} \to V_{out(h)}^{j\pm n-1 \pm d}$ for some positive integer $d$ sends $R^{-2j + n+2\mp n +(j\pm n-1)}v_2$ to $R^{-2j + n+2\mp n +(j\pm n-1)+m+2}v_3$ with $v_3\in V_{out(h)}^{j\pm n-1 \pm d}$ bounded. Finally $g_R$ sends $R^{-2j + n+2\mp n +(j\pm n-1)+m+2}v_3$ to $R^{-2j + n+2\mp n +(j\pm n-1)+d+2 +(j \pm n \pm d -1)}v_3$. In total the dependence in $R$ is given by 
\begin{equation}
-2j + n+2\mp n +(j\pm n-1)+d+2 +(j\pm n \pm d -1) \ge  2.
\end{equation}
The only cases not covered by our calculation are when both $n=d=0$ in which case we get the term $R^{-2}B_{\bar h}^0$ and a term of order $\mathcal O(1)$. 

Applying the same analysis on the other terms we obtain 
\begin{equation}
F_R = g_{A,\hbar}\cdot p_A + \mathcal O(R^2). 
\end{equation}
Therefore we may apply the inverse function theorem discussed in Section \ref{IFT} to correct this family to a family $\widetilde{F_R} \in \mathfrak M_{(0,-2i\zeta_{\R})}$ which converges to $g_{A,\hbar}\cdot p_A$. Thus we have proven:
\begin{theorem}
    Given an arbitrary point $p\in \mathfrak M_{(\zeta_{\R},0)}$ its conformal limit exists. Given that $p\in W_a^0(p_0)$ its $\hbar$-conformal limit is given by 
\begin{equation}
\begin{split}
\mathcal {CL}_{\hbar}(p) = g_{A,\hbar}\cdot(&B_h^0+A_h-\hbar(B_{\bar{h}}^0)^\d,i_k^0+I_k-{\hbar}(j_k^0)^\d,\\
&{\hbar}^{-1}(B_{\bar{h}}^0+A_{\bar{h}})+(B_h^0)^\d,\hbar^{-1}(j_k^0+J_k)+(i_k^0)^\d),
\end{split}
\end{equation}
where $A\in S^+_{p_0}$ is the unique element such that $p_0+A$ is in the gauge orbit of $p$. Finally, the conformal limit maps $W_a^0(p_0)$ biholomorphically to $W_a^1(p_0)$. 
\end{theorem}

\section{Simpson conjecture}

The aim of this last section is to introduce the analog of Simpson's conjecture \cite{Simpson} for Nakajima quiver varieties and prove it in complete generality under the usual assumption that $\zeta_{\R}$ is generic. 
\begin{theorem}
Given any $\C^{\star}$-fixed point $p_0=[(B_h^0,i_k^0,B_{\bar h}^0,j_k^0)]\in \mathfrak M_{(\zeta_{\R},0)}$, the holomorphic Lagrangian leaves $W^1_a(p_0) \subset \mathfrak M_{(0,-2i\zeta_{\R})}$ are complete in the ambient topology. 
\end{theorem}

We intentionally split the proof of this theorem into two parts. The first part deals with the case when $[p_0+A]\in W^0_a(p_0)$ is not nilpotent. The proof is closely modeled on the proof strategy developed for the original Simpson conjecture by \cite{Schulz}. We do this to emphasize the similarity between the two geometries and because this method provides sharp estimates on how fast the line $[p_0+A_R]$ escapes compact sets as $R\to \infty$. The second part deals with the case when $[p_0+A]$ is nilpotent. The proof applies to any point $[p_0+A]$ not just the nilpotent ones but it does not provide explicit quantitative control on how fast the line $[p_0+A_R]$ escapes compact sets as $R\to \infty$. At the end of the section we briefly comment on why this new proof idea is inadequate for resolving the usual Simpson conjecture. 

\subsection{Away from the nilpotent cone}
The aim of this subsection is to prove the following theorem
\begin{theorem}
Let $p_0 = [(B_h^0,i_k^0,B_{\bar h}^0,j_k^0)]$ be any $\C^{\star}$-fixed point in $\mathfrak M_{(\zeta_{\R},0)}$. Given any $A\in S^+_{p_0}$ such that $p_0+A$ is not nilpotent, the line $\mathcal{CL}_{\hbar}(p_0+A_R)$ leaves every compact subset of $\mathfrak M_{(0,-2i\zeta_{\R})}$ as $R\to \infty$. 
\end{theorem}
Before we give a proof, we need some preparation. Given a quiver $\Gamma$ with oriented edge set $H$, we augment the set $H$ by adding all edges connecting $V_k$ to $W_k$ which we denote by $\hat k$. We also include all edges connecting $W_k$ to $V_k$ and denote them as $\check k$. Let $\hat H := H\cup\{\hat 1,...,\hat n, \check 1,...,\check n\}$. A path $\mathfrak p$ is a sequence of edges $\{h_i\}_{1\le i \le m}$ in $\hat H$ such that $out(h_{i+1}) = in(h_i)$. A closed $\Gamma$-loop $\gamma$ is a path such that $h_i\in H$ for all $i$ and $in(h_m) = out(h_1)$. An admissible path is a path which starts at some $W_i$ and ends at some $W_j$. Given an admissible path $\mathfrak p$ starting at $W_k$ and ending at $W_l$ and an element $p=(B_h,i_k,B_{\bar h},j_k)\in \mathbb M$, we can associate to it the matrix $M_p(\mathfrak p):= j_m B_{h_{j-1}}... B_{h_1}i_k$ with $out(h_1) = k$ and $in(h_{j-1})= m$. Similarly for a loop $\gamma$ we can associate to it the matrix $M_p(\gamma) = B_{h_m}...B_{h_1}$. 

Given this notation, we have the following variant of the Le Bruyn--Procesi theorem \cite{LP1990},\citep[][Theorem 10.2.2. and Exercise 10.2.2.]{QuiverBook} due to Lusztig:

\begin{theorem}[\cite{L1998}, Theorem $1.3$]\label{Lusztig}
The ring of invariants $\C[\mathbb M]^{G_{\mathbf{v}}^{\C}}$ is generated by the polynomials 
\begin{itemize}
\item $\varphi(M(\mathfrak p))$ for an admissible path $\mathfrak p$ and $\varphi \in \Hom(W_k,W_l)^{\star}$. 
\item $\Tr(M(\gamma))$ for a closed path. 
\end{itemize}
\end{theorem}
\begin{proof}
First, let  us make an observation. The polynomials defined in Theorem \ref{Lusztig} are invariant under complex gauge transformations and thus only depend on the complex gauge orbit of $p$. As such we are free to use any representative from this orbit to calculate them. Now, since we have assumed that $p_0+A$ is not nilpotent, Theorem \ref{Lusztig} says that there exists either an admissible path or a loop such that the associated invariant is not zero. Without loss of generality assume it is a loop $\gamma$. Then we know that $\Tr(M_p(\gamma))\neq 0$. 

Now let us focus on the point 
\begin{equation*}
g_{A_R,\hbar}\cdot p_{A_R}
\end{equation*}
It is complex gauge equivalent to the point 
\begin{equation*}
(B_h^0+A_h - R^{-1}\hbar(B_{\bar h}^0)^{\dagger}, i_k^0+ I_k - R^{-1}\hbar(j_k^0)^{\dagger}, \hbar^{-1}R(B_{\bar h}^0 + A_{\bar h}) + (B_h^0)^{\dagger}, \hbar^{-1}R(j_k^0+J_k)+(i_k^0)^{\dagger}). 
\end{equation*}
Thus by renaming $\hbar = \hbar R^{-1}$ we need to study the behavior of the point 
\begin{equation*}
P_{A,\hbar}:= (B_h^0+A_h - \hbar(B_{\bar h}^0)^{\dagger}, i_k^0+ I_k - \hbar(j_k^0)^{\dagger}, \hbar^{-1}(B_{\bar h}^0 + A_{\bar h}) + (B_h^0)^{\dagger}, \hbar^{-1}(j_k^0+J_k)+(i_k^0)^{\dagger}). 
\end{equation*}
We consider the operator $M_{P_{A,\hbar}}(\gamma)$. Recall that the orientation $\Omega$ for our quiver does not contain any oriented cycles and therefore, at least one of the edges of the loop must lie in $\overline{\Omega}$. This implies that the leading order of $\Tr(M_{P_{A,\hbar}}(\gamma))$ in terms of $\hbar$ is given by $\hbar^{-M}\Tr(M_p(\gamma))$ where $M$ is the number of edges in $\overline{\Omega}$ appearing in the loop $\gamma$. Since $M>0$ we see that this invariant of the point $P_{A,\hbar}$ becomes unbounded as $\hbar\to 0$ and therefore the one parameter family of points $[P_{A,\hbar}]\in \mathfrak M_{(0,-2i\zeta_{\R})}$ has to escape any compact set. Therefore we are done. 
\end{proof}
\begin{remark}
If instead of a loop $\gamma$ the nontrivial invariant was given by an admissible path then this path would have to be a multiple of the matrix $j_k^0+J_k$ and therefore $M\ge1$ which would again give the required result. 
\end{remark}
\begin{remark}
Notice that these invariants function in exactly the same way as the trace of the holonomy of a complex flat connection, which was the invariant used to prove Simpson's conjecture in the Higgs bundle case \cite{Schulz,DS}.
\end{remark}

\subsection{The general case}
When $[p_0+A]$ is nilpotent, all of its associated trace invariants become zero. There is a geometric incarnation of this fact. It holds that $\mathrm{Spec}(\C[\mu_{\C}^{-1}(0)]^{G_{\mathbf{v}}^{\C}}) \cong \mathfrak M_{(0,0)}$. It is proven in \cite{Na1994} that when $\zeta_{\R}$ is generic there is a proper $\C^{\star}$-equivariant map  $\pi:\mathfrak M_{(\zeta_{\R},0)}\to \mathfrak M_{(0,0)}$ which is a resolution of singularities and such that the nilpotent cone is given by $\pi^{-1}(0)$. It is also proven that when $\zeta_{\R}$ is generic $\mathrm{Spec}(\C[\mu_{\C}^{-1}(-2i\zeta_{\R})]^{G_{\mathbf{v}}^{\C}}) \cong \mathfrak M_{(0,-2i\zeta_{\R})}$. This in particular implies that if for two points in $\mathfrak M_{(0,-2i\zeta_{\R})}$ the associated set of generators in $\C[\mu_{\C}^{-1}(-2i\zeta_{\R})]^{G_{\mathbf{v}}^{\C}}$ take the same values, then the two points must coincide. In order to complete the proof we need the following quantitative theorem of Le Bruyn and Procesi \cite{LP1990}:
\begin{theorem}
The coordinate ring $\C[\mu_{\C}^{-1}(-2i\zeta_{\R})]^{G_{\mathbf{v}}^{\C}}$ is generated by finitely many polynomials of the form given in Theorem \ref{Lusztig}. 
\end{theorem}
Notice that the point $p_A$ is gauge equivalent to 

\begin{equation*}
B +A_{\hbar^{-1}}
\end{equation*}
where 
\begin{equation*}
[B]:= [(B_h^0 - (B_{\bar h})^{\dagger}, i_k^0 - (j_k^0)^{\dagger}, B_{\bar h}^0 + (B_h^0)^{\dagger}, j_k^0+(i_k^0)^{\dagger})] \in \mathfrak M_{(0,-2i\zeta_{\R})},
\end{equation*}
and 
\begin{equation}\label{HBE}
\begin{split}
A_{\hbar^{-1}} :=& g_{\hbar^{-1}}\cdot (A_h, I_k, \hbar^{-1}A_{\bar h},\hbar^{-1}J_k)\\
=& \hbar^{-1}A_1+...+\hbar^{-M}A_m,
\end{split}
\end{equation}
where we have grouped the homogeneous in $\hbar^{-1}$ terms together, so that $A_i$ do not depend on $\hbar$. We want to show that as $\hbar\to 0$ the line $\mathcal CL_{\hbar}(p_0+A)$ leaves every compact set of $\mathcal M_{(0,-2i\zeta_{\R})}$. Working in the above gauge, it is evident that each invariant polynomial from Theorem \ref{Lusztig} has to be a polynomial in $\hbar^{-1}$. Since the conformal limit is a biholomorphism, varying $\hbar$ gives different points in $\mathcal M_{(0,-2i\zeta_{\R})}$ and therefore, at least one of the invariant polynomials is not constant. This immediately implies that as $\hbar\to 0$ this polynomial has to become unbounded. However, this immediately implies that as $\hbar \to 0$ we must be leaving every bounded set of $\mathrm{Spec}(\C[\mu_{\C}^{-1}(-2i\zeta_{\R})]^{G_{\mathbf{v}}^{\C}})$ since generators take bounded values on bounded sets. Therefore we are done. 

\bibliography{bibliography}
\bibliographystyle{alpha}

\vspace{12pt}
\noindent
Sotiria Chatzimarkou, Department of Mathematics, University of Patras\\
University campus, Patras 26504, Greece.\\
\textit{sotiria07gr@gmail.com}

\vspace{7pt}
\noindent
Panagiotis Dimakis, Department of Mathematics,University of Maryland\\
College Park 20740, MD, USA.\\
\textit{pdimakis12345@gmail.com}

\end{document}